%% file: Disc.tex
\documentclass{lms}

\usepackage{enumerate}
\usepackage{amsmath}
\usepackage{amssymb}
\usepackage{epsfig}
\usepackage[matrix,arrow,curve]{xy}
\usepackage{color,graphicx}

\usepackage[breaklinks,bookmarksopen,bookmarksnumbered]{hyperref}

\providecommand{\Aut}{\mathop{\rm Aut}\nolimits}
\providecommand{\Bir}{\mathop{\rm Bir}\nolimits}
\providecommand{\Card}{\mathop{\rm Card}\nolimits}
\providecommand{\Disc}{\mathop{\rm Disc}\nolimits}
\providecommand{\Conv}{\mathop{\rm Conv}\nolimits}
\providecommand{\id}{\mathop{\rm Id}\nolimits}

\providecommand{\interieur}{\mathop{\rm Int}\nolimits}
\providecommand{\lc}{\mathop{\rm lc}\nolimits}
\providecommand{\ord}{\mathop{\rm ord}\nolimits}
\providecommand{\Res}{\mathop{\rm Res}\nolimits}
\providecommand{\Sing}{\mathop{\rm Sing}\nolimits}
\providecommand{\Supp}{\mathop{\rm Supp}\nolimits}

\newtheorem{theorem}{Theorem}[section]
\newtheorem{lemma}[theorem]{Lemma}
\newtheorem{proposition}[theorem]{Proposition}
\newtheorem{corollary}[theorem]{Corollary}

\newtheorem{remark}[theorem]{Remark}
\newtheorem{definition}[theorem]{Definition}

\newtheorem{Problem}[theorem]{Problem}
\newtheorem{question}[theorem]{Question}

\title{Plane Curves with Minimal Discriminant}

\author{D. Simon and M. Weimann} 

\classno{14H50 (primary), 11R29, 13P15, 14E07, 14H20, 14H45 (secondary)}

\begin{document}
\maketitle

\begin{abstract}
We give lower bounds for the degree of the discriminant with respect to $y$ of  separable polynomials $f\in \mathbb{K}[x,y]$ over an algebraically closed field of characteristic zero. Depending on the invariants involved in the lower bound, we give a geometrical characterisation of those polynomials having minimal discriminant, and give  an explicit construction of all such polynomials in many cases. In particular, we show that irreducible monic polynomials with minimal discriminant coincide with coordinate polynomials. We obtain analogous partial results for the case of nonmonic or reducible polynomials by studying their $GL_2(\mathbb{K}[x])$--orbit and by establishing some combinatorial constraints on their Newton polytope. Our results suggest some natural extensions of the embedding line theorem of Abhyankar-Moh and of the Nagata-Coolidge problem to the case of unicuspidal curves  of $\mathbb{P}^1\times\mathbb{P}^1$.
\end{abstract}



\section{Introduction}\label{s0}

Let $f\in \mathbb{K}[x,y]$ be a bivariate polynomial defined over an algebraically closed field $\mathbb{K}$ of characteristic zero. We denote by $d_x$ and $d_y$ the respective partial degrees of $f$ with respect to $x$ and $y$, and by 
\[
\Delta_y(f):=\Disc_{y}(f) \in \mathbb{K}[x]
\]
the discriminant of $f$ with respect to $y$. In this note, we study polynomials with discriminants of low degrees. More precisely, we focus on the following problem:

\begin{Problem}\label{P1}
Give a lower bound for the degree of the discriminant in terms of some invariants attached to $f$ and construct all polynomials whose discriminant reaches this lower bound.
\end{Problem}

Throughout the paper, we assume that $f$ is \textit{primitive} (with respect to $y$), that is $f$ has no factor in $\mathbb{K}[x]$. This hypothesis is not restrictive for our purpose thanks to the well known multiplicative properties of the discriminant. We also assume that $f$ is separable with respect to $y$ in order to avoid zero discriminants.

\paragraph*{The case of monic polynomials.}  We say that $f\in \mathbb{K}[x,y]$ is \textit{monic} (with respect to $y$) if its leading coefficient, with respect to $y$, is invertible, that is does not depend on $x$. 


\begin{theorem}\label{tminimal}
Let $f\in \mathbb{K}[x,y]$ be a primitive squarefree polynomial with $r$ irreducible factors. Then
\[
\deg_x \Delta_y(f)\ge d_y-r.
\] 
If moreover $f$ is monic, then the equality holds if and only if there exists a polynomial automorphism $\sigma=(\sigma_x,\sigma_y)\in \Aut(\mathbb{A}^2)$ and a degree $r$ polynomial $g\in \mathbb{K}[y]$ such that
$f=g\circ \sigma_y$. 
\end{theorem}


The group $\Aut(\mathbb{A}^2)$ of automorphisms of $\mathbb{A}^2$ being generated by affine and elementary automorphisms thanks to Jung's Theorem \cite{Jung}, Theorem \ref{tminimal} gives a solution of Problem \ref{P1} for monic polynomials in terms of the invariants $d_y$ and $r$. Moreover, given $f$ monic for which the equality holds, we can compute the automorphism $\sigma$ recursively from the Newton polytope of any irreducible factor of $f$. 

Theorem \ref{tminimal} implies in particular that if $f$ is monic and satisfies $\deg_x \Delta_y(f)= d_y-r$, then $r$ divides $d_y$. Hence, either its discriminant is constant, or it satisfies the inequality
\[
\deg_x \Delta_y(f)\ge \Big\lceil \frac{d_y-1}{2} \Big\rceil.
\]
It turns out that this fact is still true for nonmonic polynomials, and we have moreover a complete classification of polynomials for which equality holds, solving Problem \ref{P1} in terms of the invariant $d_y$. The precise result requires some more notation and will be stated later in this introduction (Theorem \ref{treducible}).

Thanks to the multiplicative properties of the discriminant, the inequality in Theorem \ref{tminimal} is equivalent to the fact that any \textit{irreducible} polynomial satisfies the inequality \[\deg_x\Delta_y(f) \ge d_y-1.\]
A similar lower bound for irreducible polynomials appears in \cite[Prop.\,1]{CK}, under the additional assumption that $\deg f=d_y$. The second part of Theorem \ref{tminimal} for $r=1$ has to be compared with \cite[Thm.\,4]{GP}, where the authors show that \textit{if $d_y$ coincides with the total degree of $f$}, then $f$ is a coordinate of $\mathbb{C}^2$ if and only if $f$ is a Jacobian polynomial such that $\deg_x\Delta_y(f)=d_y-1$. Our result allows to replace the Jacobian hypothesis by irreducibility. Note further that being monic is a weaker condition than $\deg f=d_y$. 

\paragraph*{Bounds with respect to the genus.} If now we take into account the genus $g$ and the degree $d_y$ of $f$, we can refine the lower bound $d_y-1$ for irreducible polynomials: 

\begin{theorem}\label{t1}
Let $f\in \mathbb{K}[x,y]$ be a primitive irreducible polynomial. Then
\[ 2g+d_y-1 \le \deg_x\Delta_y(f) \le 2d_x(d_y-1),\]
where $g$ stands for the geometric genus of the algebraic curve defined by $f$. 
Moreover, the equality
\[ \deg_x\Delta_y(f) = 2g+d_y-1 \]
holds if and only if the Zariski closure $C\subset \mathbb{P}^1\times \mathbb{P}^1$ of the affine curve $f=0$ is a genus $g$ curve with a unique place supported on the line $x=\infty$ and smooth outside this place. 
\end{theorem}

Theorem \ref{tminimal}  is mainly a consequence of Theorem \ref{t1} combined with the embedding line theorem of Abhyankar-Moh \cite{AM} that asserts that every embedding of the line in the affine plane $\mathbb{A}^2$ extends to a polynomial automorphism of the plane. In particular, it appears the remarkable fact that a monic irreducible polynomial with  minimal discriminant with respect to $y$ is  also monic with minimal discriminant with respect to $x$ (Theorem \ref{t3}).

\begin{remark} Our results are specific to fields of characteristic zero. 
For instance, if $\mathbb{K}$ has characteristic $p$, the polynomial $f(x,y)=y^p+y^k+x$ is irreducible and satisfies
\[
\deg_x \Delta_y(f)=k-1,\quad \forall\,\, 1\le k < p.
\]
Hence, there is no nontrivial lower bound for the degree of the discriminant if we do not take some care on the degree. 
\end{remark}

\paragraph*{$G$-reduction of (nonmonic) minimal polynomials.}
We say that $f\in \mathbb{K}[x,y]$ is \textit{minimal} if it is irreducible and if its discriminant reaches the lower bound
\[\deg_x \Delta_y(f) = d_y-1.\]
Theorem \ref{tminimal} characterises monic minimal polynomials: they coincide with \textit{coordinate} polynomials, that is polynomials that form part of a basis of the $\mathbb{K}$-algebra $\mathbb{K}[x,y]$. 
In the nonmonic case, the characterisation of minimal polynomials is more complicated. Indeed, the second part of Theorem \ref{tminimal} is false in general since $\Aut(\mathbb{A}^2)$ does not preserve minimality of nonmonic polynomials.  An idea is to introduce other group actions in order to reduce minimal polynomials to a "canonical form". Since the discriminant of $f$ coincides with the discriminant of its homogenisation $F$ with respect to $y$, 
we may try to apply a reduction process to $F$.  The multiplicative group $G:=GL_2(\mathbb{K}[x])$ acts on the space $\mathbb{K}[x][Y]$ of homogeneous forms in $Y=(Y_0:Y_1)$ with coefficients in $\mathbb{K}[x]$ by
\begin{equation}\label{action1}
\begin{pmatrix} a & b \\ c & d \end{pmatrix}(F)=F(aY_0+b Y_1,cY_0+dY_1). 
\end{equation}
The partial degree $d_Y$ of $F$, the number $r$ of irreducible factors and the degree of the discriminant are $G$-invariant (see Section \ref{s1}). The group $G$ is thus a good candidate for reducing nonmonic polynomials with small discriminant to a simpler form, in the same vein as in Theorem \ref{tminimal}. 

We say that $F,H\in \mathbb{K}[x][Y]$ are $G$--equivalent, denoted by $F\equiv H$, if there exists $\sigma\in G$ such that $F=\sigma(H)$. The action (\ref{action1})  induces by dehomogenisation a well defined action on the set of \textit{irreducible} polynomials in $\mathbb{K}[x,y]$ with $d_y>1$, and more generally on the set of polynomials \textit{with no linear factors} in $y$. In particular, we can talk about $G$--equivalence of (affine) minimal polynomials of degree $d_y>1$. 

The $G$-orbit of a monic minimal polynomial contains many nonmonic minimal polynomials and it is natural to ask if all nonmonic minimal polynomials arise in such a way. We prove that the answer is no in general thanks to the following  counterexample.

\begin{theorem}\label{contrexintro}
Let $\lambda\in \mathbb{K}^*$. The polynomial $f= x(x-y^2)^2-2\lambda y(x-y^2)+\lambda^2$ is minimal but is not $G$--equivalent to a monic polynomial.
\end{theorem}

This result will follow as a corollary of the $G$-reduction Theorem \ref{reduced} which shows in particular that if the degree $c$ of the leading coefficient of a minimal polynomial is not the smallest in the $G$-orbit, then $d_y$ necessarily divides $d_x-c$. The proof is in the spirit of Wightwick's results  \cite{wight} about orbits of $\Aut(\mathbb{C}^2)$. 
Although we can guess that this example is not unique, we were not able to find a single other such example despite a long computer search (see Subsection \ref{ssparam}). Indeed, it turns out that being simultaneously minimal and $G$-reduced still imposes divisibility restrictions on the partial degrees. In particular, we can show that all minimal polynomials of prime degree $d_y$ are $G$--equivalent to a monic polynomial, solving Problem \ref{P1} in that context. More precisely:

\begin{theorem}\label{nonmonicintro} Let $f$ be a minimal polynomial of prime degree $d_y$. Then there exists  $g\in \mathbb{K}[y]$ of degree $d_y$ such that 
\[
f(x,y)\equiv g(y)+x.
\]
In particular, $f$ is $G$--equivalent to a monic polynomial, hence to a coordinate polynomial.
\end{theorem}

Theorem \ref{nonmonicintro} follows from the fact that minimality implies that either $d_y$ divides $d_x-c$ or $d_x-c$ and $d_y$ are not coprime except for some trivial cases (Theorem \ref{conj}). The proof relies on a suitable toric embedding of the curve of $f$. 
It is natural to ask whether minimality implies the stronger 
fact that either $d_y$ divides $d_x-c$ or $d_x-c$ divides $d_y$. 
This property holds for $c=0$, a statement equivalent 
to the Abhyankar-Moh Theorem \cite{Ab}. In general, we need to study the singularity of smooth rational curves of $\mathbb{A}^1\times \mathbb{P}^1$ with a unique place along $\infty \times\mathbb{P}^1$, generalising the Abhyankar-Moh situation of smooth rational curves of $\mathbb{A}^2$ with a unique place at the infinity of $\mathbb{P}^2$. 

\paragraph*{Cremona equivalence of minimal polynomials.} In a close context, we can pay attention to Cremona reduction of minimal polynomials. Theorem \ref{contrexintro} shows that it is hopeless to reduce a nonmonic minimal polynomial to a coordinate by applying successively $GL_2(\mathbb{K}[x])$ and $\Aut(\mathbb{A}^2)$. However, both groups can be considered as subgroups of the Cremona group $\Bir(\mathbb{A}^2)$ of birational transformations of the plane and our results suggest to ask whether all minimal polynomials define curves that are Cremona equivalent to a line. We will prove for instance that the nonmonic minimal polynomial in Theorem \ref{contrexintro} satisfies this property (Proposition \ref{prop:exemple}). This open problem can be seen as a  generalisation of the Coolidge-Nagata problem \cite{KP} to unicuspidal curves of $\mathbb{P}^1\times \mathbb{P}^1$. 


\paragraph*{A uniform lower bound for reducible polynomials.} 
Our last result gives a uniform sharp lower bound for the degree of the discriminant of any separable (reducible) polynomial that depends only on $d_y$. Moreover it establishes a complete classification of polynomials that reach this lower bound. We need to express this classification in homogeneous coordinates, and we let $F\in \mathbb{K}[x][Y]$ stands for the homogeneous form associated to $f$ of degree $\deg_Y F=d_y$.

\begin{theorem}\label{treducible}
Let $f\in \mathbb{K}[x,y]$ be a primitive and squarefree polynomial. Then $f$ has constant discriminant if and only if $F\equiv H$ for some $H\in \mathbb{K}[Y]$. Otherwise, we have the inequality 
\[
\deg_x \Delta_y(f)\ge \Big\lceil \frac{d_y-1}{2} \Big\rceil
\] 
and the equality holds if and only if one of the following conditions holds:
\begin{enumerate}
\item $d_y=4$ and $F\equiv Y_0Y_1(Y_0^2+(\mu x+\lambda)Y_0Y_1+ Y_1^2)$, with $\mu,\lambda\in \mathbb{K}$, $\mu\ne 0$. 
\item $d_y=4$ and $F\equiv Y_1(H(Y)+xY_1^3)$, for some cubic form $H\in \mathbb{K}[Y]$.
\item $d_y$ is odd and $F\equiv Y_1 H(Y_0^2+xY_1^2,Y_1^2)$ for some form $H\in \mathbb{K}[Y]$.
\item $d_y$ is even and $F\equiv H(Y_0^2+xY_1^2,Y_1^2)$ for some form $H\in \mathbb{K}[Y]$.
\end{enumerate}
\end{theorem}


\paragraph*{Organisation of the paper.} We prove Theorem \ref{t1} in Section \ref{s1}. The proof is based on the classical relations between the valuation of the discriminant and the Milnor numbers of the curve along the corresponding critical fiber.  We prove Theorem  \ref{tminimal} in Section \ref{s2}, the main ingredients of the proof being Theorem \ref{t1} combined with the embedding line theorem of Abhyankar-Moh. In particular, we show that for a monic polynomial, minimality with respect to $y$ is equivalent to minimality with respect to $x$ (Theorem \ref{t3}). In Section \ref{SSnonmonic}, we focus on the $GL_2(\mathbb{K}[x])$-orbits of nonmonic minimal polynomials. We first characterise minimal polynomials that minimise the volume of the Newton polytope in their orbit (Subsection \ref{orbits}, Theorem \ref{reduced}). The counterexample of Theorem \ref{contrexintro} follows as a corollary. Although this example is not $G$--equivalent to a coordinate, we show in Subsection \ref{Cremona} that it defines a curve Cremona equivalent to a line and we address the question if this property holds for all minimal polynomials. In a close context, we show in Subsection \ref{toric} that the partial degrees of minimal polynomials obey to some strong divisibility constraints (Theorem \ref{conj}). Theorem \ref{nonmonicintro} follows as a corollary. At last, we prove Theorem \ref{treducible} in Section \ref{Sfin}.
The paper finishes with three appendices on related problems. In Appendix \ref{s3}, we study the relations between small discriminants with respect to $x$ and small discriminants with respect to $y$, extending Theorem \ref{t3} to the nonmonic case. In Appendix \ref{ssparam}, we give a parametric characterisation of minimal polynomials and we apply our result to the Computer Algebra challenge of computing nonmonic minimal polynomials. Finally, we give in Appendix  \ref{C}  a direct and instructive proof of the fact that coordinate polynomials are minimal, some of the lemmas listed here being used in the main part of the paper.

\section{Bounds for the degree of the discriminant. Proof of Theorem \ref{t1}.}\label{s1}

The upper bound in Theorem \ref{t1} for the degree of the discriminant follows from classical results about the partial degrees of discriminants of homogeneous forms with indeterminate coefficients. The lower bound follows by studying the relations between the vanishing order of the discriminant at infinity and the singularities of the curve of $f$. 

\vskip2mm
\noindent

We recall that in all of the sequel, $f$ is assumed to be primitive, hence with no factors in $\mathbb{K}[x]$. This assumption is not restrictive for our purpose thanks to the well known formula $\Delta_y(uf)=u^{2d_y-2}\Delta_y(f)$ when $u\in \mathbb{K}[x]$.

\paragraph*{Bihomogenisation.} Let us denote by $F$ the bihomogenised polynomial of $f$
\[
F(X,Y):= X_1^{d_x}Y_1^{d_y} f \Big(\frac{X_0}{X_1},\frac{Y_0}{Y_1}\Big),
\] 
where $X=(X_0:X_1)$ and $Y=(Y_0:Y_1)$ are homogeneous variables. We write $\mathbb{K}[X,Y]$ for the space of bihomogeneous polynomials in $X$ and $Y$. We denote by 
\[
\Delta_Y(F):=\Disc_Y(F)
\]
the discriminant of $F$ seen as a homogeneous form in $Y$. It is a homogeneous polynomial of degree $2d_y-2$ in the coefficients of $F$ which vanishes if and only if $F$ is not separable with respect to $Y$. In our case, it follows that $\Delta_Y(F)$ is a homogeneous polynomial in $X$ of total degree
\[
\deg_X\Delta_Y(F)=2d_X(d_Y-1)=2d_x(d_y-1).
\]
Since dehomogenisation of the discriminant of $F$ coincides with the discriminant of $f$, we get the following relation
\[
\deg_x \Delta_y(f)=\deg_X \Delta_Y(F) - \ord_{\infty} \Delta_Y(F)
\]
where $\ord_{\infty}$ stands for the vanishing order at $\infty:=(1:0)\in \mathbb{P}^1$. 
The upper bound in Theorem \ref{t1} follows. In order to get the lower bound, one needs an upper bound for $\ord_{\infty}\Delta_Y(F)$. Let 
\[
C\subset \mathbb{P}^1\times \mathbb{P}^1
\] 
be the curve $F=0$. It coincides by construction with the Zariski closure of the affine curve $f=0$ in the product of projective spaces $\mathbb{P}^1\times \mathbb{P}^1$. For a point $\alpha\in \mathbb{P}^1$ we denote by 
\[
Z_{\alpha}:= C\cap (X=\alpha)
\]
the set theoretical intersection of $C$ with the "vertical line" $X=\alpha$. 
It is zero-dimensional since otherwise $F$ would have a linear factor in $X$, contradicting the primitivity assumption on $f$. Moreover, we have 
\[
\Card(Z_{\alpha})\le d_Y,
\]
with strict inequality if and only if $\Delta_Y(F)(\alpha)=0$, that is if and only if $F(\alpha,Y)$ is not squarefree. In order to understand the order of vanishing of $\Delta_Y(F)$ at $\alpha$, we need to introduce some classical local invariants of the curve $C$.

\paragraph*{The ramification number.} Let $p\in C$. A \textit{branch} of $C$ at $p$ is an irreducible analytic component of the germ of curve $(C,p)$. 

\begin{definition}\label{def1} The \textit{ramification number} of $C$ over $\alpha\in \mathbb{P}^1$ is defined as
\[
r_{\alpha}:=d_y-\sum_{p\in Z_{\alpha}} n_p,
\]
where $n_p$ stands for the number of branches  of $C$ at $p$. 
\end{definition}

In other words, the ramification number measures the defect to the expected number $d_Y$ of branches of $C$ along the vertical line $X=\alpha$. It is also equal to the sum $\sum (e_{\beta}-1)$ over all places $\beta$ of $C$ over $\alpha$, where $e_{\beta}$ stands for the ramification index of $\beta$.

\paragraph*{The delta invariant.} Let $B$ be a branch. The local ring $\mathcal{O}_B$ has finite index in its integral closure $\bar{\mathcal{O}}_{B}$. The quotient ring is a finite dimensional vector space over $\mathbb{K}$ whose dimension 
\[
\delta(B):=\dim_{\mathbb{K}} \,\,\bar{\mathcal{O}}_{B}/\mathcal{O}_{B}
\]
is called the \textit{delta invariant} of $B$. More generally, we define the \textit{delta invariant} of $C$ at $p$ as the nonnegative integer
\[
\delta_p(C):=\sum_i \delta_p(B_i) +\sum_{i<j} (B_i\cdot B_j)_p
\]
where the $B_i$'s run over the branches of $C$ at $p$ and where $(B_i\cdot B_j)_p$ stands for the intersection multiplicity at $p$ of the curves $B_i$ and $B_j$. In some sense, the delta invariant $\delta_{p}(C)$ measures the complexity of the singularity of $C$ at $p$. In particular, we have  $\delta_p(C)=0$ if and only if $C$ is smooth at $p$.

\begin{definition}\label{def2}
The \textit{delta invariant of $C$ over $\alpha\in \mathbb{P}^1$} is
\[
\delta_{\alpha}:=\sum_{p\in Z_{\alpha}} \delta_{p}(C).
\] 
\end{definition}

The integer $\delta_{\alpha}$ thus measures the complexity of all singularities of $C$ that lie over $\alpha$.

\paragraph*{$PSL_2$-invariance of the discriminant.} The multiplicative group $GL_2(\mathbb{K}[x])$ of $2\times 2$ invertible matrices with coefficients in $\mathbb{K}[x]$ acts naturally on the space $\mathbb{K}[x][Y]$ of homogeneous forms in $Y=(Y_0:Y_1)$ with coefficients in $\mathbb{K}[x]$  by
\begin{equation}\label{action}
\begin{pmatrix} a & b \\ c & d \end{pmatrix}(F)=F(aY_0+b Y_1,cY_0+dY_1)
\end{equation}
This action preserves the degree in $Y$ and for $\tau \in GL_2(\mathbb{K}[x])$, we have 
\begin{equation}\label{invariance}
\Delta_Y(\tau(F))=\det(\tau)^{d_Y(d_Y-1)}\Delta_Y(F),
\end{equation}
so that the discriminant is $PSL_2(\mathbb{K}[x])$-invariant and the degree of the discriminant is $GL_2(\mathbb{K}[x])$-invariant.  This action also preserves the irreducibility. 
It induces by dehomogenisation a well defined action on the set of \textit{irreducible} polynomials in $\mathbb{K}[x,y]$ with $d_y>1$, and more generally on the set of polynomials \textit{with no linear factors} in $y$. The corresponding formula is  
\begin{equation}\label{action2}
\begin{pmatrix} a & b \\ c & d \end{pmatrix}(f)=(cy+d)^{d_y} f\Big(x,\frac{ay+b}{cy+d}\Big).
\end{equation}
We will study in more details the action of $GL_2(\mathbb{K}[x])$ in Section \ref{orbits}.

\paragraph*{Vanishing order of the discriminant.} For $\alpha=(\alpha_0:\alpha_1)\in \mathbb{P}^1$ and $H\in \mathbb{K}[X_0:X_1]$ a homogeneous form, the vanishing order $\ord_{\alpha} H$ of $H$ at $\alpha$ is the highest power of $\alpha_0 X_1 - \alpha_1 X_0$ that divides $H$. The vanishing order at $\alpha\ne \infty$ coincides with the usual valuation of the dehomogenisation of $H$ at $x-\alpha$. The vanishing order of the discriminant is related to the  ramification degree and the delta invariant thanks to the following key proposition:

\begin{proposition}\label{p1}
Let $\alpha\in \mathbb{P}^1$ and $F\in \mathbb{K}[X,Y]$ a bihomogeneous form with no factors in $\mathbb{K}[X]$. We have the equality \[\ord_{\alpha} \Delta_Y(F)=r_{\alpha} +2\delta_{\alpha}.\] In particular, we have
\[
\deg_{x} \Delta_y(f)=2d_x(d_y-1)-2\delta_{\infty}-r_{\infty}.
\]
\end{proposition}

\begin{proof}
Up to a change of coordinates of $\mathbb{P}^1$, there is no loss to assume that $\alpha=(0:1)$ and we will write simply $\ord_0$ for $\ord_{(0:1)}$. Note first that $\ord_0\Delta_Y(F)=\ord_0 \Delta_y(f)$. Since $\mathbb{K}$ has infinite cardinality, there exists $\beta \in \mathbb{K}$ such that $f(0,\beta)\ne 0$. For such a $\beta$, the leading coefficient with respect to $y$ of the transformed polynomial $y^{d_y}f(x,\beta+1/y)$ is a unit modulo $x$. Since by (\ref{invariance}) the discriminant is invariant under $PSL_2(\mathbb{K})$, we can thus assume that the leading coefficient of $f$ with respect to $y$ is a unit modulo $x$, meaning that the point $(0,\infty)$ does not belong to $C$. In such a case, Hensel's lemma ensures that we have a unique factorisation
\[
f=u\prod_{p\in Z_0} f_p \in \mathbb{K}[[x]][y]
\]
where $u\in\mathbb{K}[[x]]$ is a unit and where $f_p\in \mathbb{K}[[x]][y]$ is a monic polynomial giving the equation of the germ of curve $(C,p)$. Note that $f_p$ is not necessary irreducible. By the well known multiplicative relations between discriminants and resultants, we have 
\[
\Delta_y(f)=\pm u^{2d_y-2}\prod_{p\in Z_0} \Delta_{f_p}\prod_{p\ne q}\Res_y(f_p,f_q)
\]
where $\Res_y$ stands for the resultant with respect to $y$. The roots of $f_p(0,y)$ and $f_q(0,y)$ are distinct by assumption so the resultant $\Res(f_p,f_q)$ is a unit in $\mathbb{K}[[x]]$. Hence,
\[
\ord_0\Delta_y(f)=\sum_{p\in Z_0} \ord_0\Delta_y(f_p).
\] 
Since $f_p$ is a distinguished polynomial, we have
\[
\ord_0 \Delta_y(f_p)=(C\cdot C_y)_p,
\]
where $C_y$ stands for the polar curve $\partial_y f=0$. Now by Teissier's Lemma \cite[Chap.\,II, Prop.\,1.2]{tei}, we have
\[
(C\cdot C_y)_p=\mu_p + d_p-1
\]
where $d_p$ stands for the degree in $y$ of $f_p$ and where $\mu_p$ stands for the \textit{Milnor number} of $C$ at $p$, that is
\[
\mu_p(C):=(C_x\cdot C_y)_p,
\]
with $C_x$ the polar curve $\partial_x f=0$. The Milnor number and the delta invariant of a germ of curve are related by the Milnor-Jung formula {\cite[Thm.\,6.5.9]{wall}}
\[
\mu_p(C)=2\delta_p(C)-n_p(C)+1,
\]
where $n_p(C)$ stands for the number of branches of $C$ at $p$. 
Finally, we get:
\[
\ord_0 \Delta_y(f)=\sum_{p\in Z_0} (2\delta_p(C)-n_p(C)+d_p)
\]
Proposition \ref{p1} then follows from equality $\sum_p d_p=d_y$. 
\end{proof}

\paragraph*{Adjunction formula.} For $C$ an irreducible algebraic curve on a smooth complete algebraic surface, the \textit{adjunction formula} asserts that the difference between the arithmetic genus $p_a(C)$ and the geometric genus $g(C)$ is equal to the total sum of the delta invariants of the curve, that is
\begin{equation}\label{adjunction}
p_a(C)=g(C) + \sum_{p\in \Sing(C)} \delta_p(C),
\end{equation}
see for instance {\cite[Sec.\,2.11]{barth}}. This formula generalises the famous Pl\"ucker formula that computes the geometric genus of a projective plane curve with ordinary singularities. We deduce the following bound for the valuation of the discriminant:

\begin{proposition}\label{p2}
Let $\alpha\in \mathbb{P}^1$ and $F\in \mathbb{K}[X,Y]$ an irreducible bihomogeneous polynomial of partial degree $d_Y>0$ and geometric genus $g$. We have the inequality
\[
\ord_{\alpha}\Delta_Y(F)\le (2d_X-1)(d_Y-1)-2g.
\]
Moreover, equality holds if and only if the curve $C\subset \mathbb{P}^1\times \mathbb{P}^1$ defined by $F=0$ has a unique place on the line $X=\alpha$ and is smooth outside this place.
\end{proposition}

\begin{proof}
Since $C$ has at least one branch along the line $X=\alpha$, the ramification number $r_{\alpha}$ is bounded above by $d_y-1$. Hence Proposition \ref{p1} implies that
\[
\ord_{\alpha}\Delta_Y(F)\le 2\sum_{p\in \Sing(C)} \delta_p(C)+d_y-1.
\]
It is well known that a curve $C\subset \mathbb{P}^1\times \mathbb{P}^1$ of bidegree $(d_x,d_y)$ has arithmetic genus 
\[
p_a(C)=(d_x-1)(d_y-1)
\]
and the upper bound of Proposition \ref{p2} follows from the adjunction formula (\ref{adjunction}). 
Equality holds in Proposition \ref{p2} if and only if both invariants $\delta_{\alpha}$ and $r_{\alpha}$ are maximal once the genus is fixed. This is equivalent to the equalities
\[
\delta_{\alpha}=\sum_{p\in \Sing(C)} \delta_p(C)  \quad {\rm and} \quad r_{\alpha}=d_y-1.
\]
The first equality is equivalent to $\delta_{\beta}=0$ for all $\beta\ne \alpha$, meaning geometrically that $C$ is smooth outside the line $X=\alpha$. The second equality is equivalent to the fact that $C$ has a unique branch along this line.
\end{proof}

\paragraph*{Proof of Theorem \ref{t1}.} Theorem \ref{t1} follows by combining the equality
\[
\deg_x \Delta_y(f)=\deg_X\Delta_Y(F) - \ord_{\infty} \Delta_Y(F)
\]
with the inequality of Proposition \ref{p2}. $\quad\square$

\begin{corollary}\label{cmin}
Let $f\in \mathbb{K}[x,y]$ be an irreducible polynomial of partial degree $d_y>0$. Then 
\[
\deg_x\Delta_y(f)\ge d_y-1
\]
and equality holds if and only if the curve 
$C\subset \mathbb{P}^1\times \mathbb{P}^1$
is rational, with a unique place over the line $x=\infty$, and smooth outside this place.
\end{corollary}

\paragraph*{Almost minimal discriminants.} Thanks to a parity reason, we can give also a geometrical characterisation of polynomials with "almost minimal" discriminant, that is for which equality $\deg_x \Delta(f)=d_y$ holds.

\begin{corollary}\label{pd_y}
Let $f\in\mathbb{K}[x,y]$ be irreducible. Then equality \[\deg_x \Delta_y(f)=d_y\] holds if and only if the closed curve $C\subset \mathbb{P}^1\times \mathbb{P}^1$ defined by $f$ is rational, with two places over the line $x=\infty$ and smooth outside these places.
\end{corollary}

\begin{proof} By Proposition \ref{p1}, we have $\deg_x \Delta_y(f)=d_y$ if and only if 
\[
d_y=2d_x(d_y-1)-2\delta_{\infty}-r_{\infty}.
\]
Since $\delta_{\infty}\le (d_x-1)(d_y-1)$ by the adjunction formula, it follows that 
$r_{\infty}\ge d_y-2$. But we have $r_{\infty}\le d_y-1$ and equality can not hold for a parity reason. Hence the only solution is $r_{\infty}=d_y-2$ and $\delta_{\infty}= (d_x-1)(d_y-1)$. This exactly means that $C$ is rational with two places over the line $x=\infty$ and smooth outside these two places. \end{proof}

\section{Classification of minimal monic polynomials. Proof of Theorem \ref{tminimal}}\label{s2}

\begin{definition}
We say that $f\in \mathbb{K}[x,y]$ is \textit{minimal} (with respect to $y$) if it is irreducible and satisfies the equality $\deg_x\Delta_y(f)= d_y-1$.
\end{definition}

\begin{definition} We say that $f\in \mathbb{K}[x,y]$ is \textit{monic} with respect to $y$ (resp. to $x$) if its leading coefficient with respect to $y$ (resp. to $x$) is constant. Take care that in the literature, this terminology often refers to polynomials with leading coefficient equal to $1$. 
\end{definition}

\subsection{Characterisation of monic minimal polynomial.}

\begin{theorem}\label{t3}
Let $f\in \mathbb{K}[x,y]$ be a nonconstant irreducible bivariate polynomial. The following assertions are equivalent:
\begin{enumerate}[(a)]
\item $d_y=0$, or $\deg_x\Delta_y(f)=d_y-1$ and $f$ is monic with respect to $y$.
\item $d_x=0$, or $\deg_y\Delta_x(f)=d_x-1$ and $f$ is monic with respect to $x$.
\item The affine curve $f=0$ is smooth rational, and has a unique place at infinity of $\mathbb{P}^2$.
\item There exists $\sigma\in\Aut(\mathbb{A}^2)$ such that $f\circ \sigma=y$.
\end{enumerate}
\end{theorem}

Thanks to Jung's Theorem \cite{Jung}, we have an explicit description of the group $\Aut(\mathbb{A}^2)$ of polynomial automorphisms of the plane. Namely, it is generated by the transformations $(x,y)\to (y,x)$ and $(x,y)\to (x,\lambda y +p(x))$ with $\lambda\in \mathbb{K}^*$ and $p\in \mathbb{K}[x]$. Hence Theorem \ref{t3} gives a complete and explicit description of all minimal monic polynomials. 
Note the remarkable fact that for monic polynomials, minimality with respect to $y$ is equivalent to minimality with respect to $x$. This symmetry can be extended to nonmonic polynomials by taking into account the number of roots of the leading coefficients, see Appendix \ref{s3}.

\begin{proof}
$(a)\Rightarrow (b)$. If $d_x=0$, the assertion is trivial. If $d_y=0$ then by the irreducibility assumption, we have $f=ax+b$ for some $a\in\mathbb{K}^*$ and $b\in \mathbb{K}$ so that $(b)$ trivially holds too. Suppose now that $d_y>0$ and $d_x>0$.  By Theorem \ref{t1}, the curve $C\subset \mathbb{P}^1\times \mathbb{P}^1$ defined by $f$ has a unique place $p$ on the line $x=\infty$ and is smooth outside this place. 
Since $f$ is supposed to be monic with respect to $y$ and $d_x>0$, the curve $C$ intersects the line $y=\infty$ at the unique point $(\infty,\infty)$. This forces equality  $p=(\infty,\infty)$. Hence $C$ is rational with a unique place over the line $y=\infty$ and smooth outside this line. Thus $f$ has minimal discriminant with respect to $x$ by Theorem \ref{t1}. Since $C$ has a unique place on the divisor at infinity  $B:=\mathbb{P}^1\times \mathbb{P}^1\setminus \mathbb{A}^2$, usual arguments (see Lemma \ref{pol}) ensure that the Newton polytope of $f$ has an edge that connects the points $(d_x,0)$ and $(0,d_y)$. In particular, $f$ is necessarily monic with respect to $x$.

$(b)\Rightarrow (a)$. Follows by the symmetric roles played by the variables $x$ and $y$.

$(a)\Leftrightarrow (c)$. If $d_y=0$ or $d_x=0$, then the result is trivial. Suppose now that $d_x$ and $d_y$ are positive. We just saw that this is equivalent to the fact that $C$ is rational, with $(\infty,\infty)$ as unique place on the divisor at infinity  $B:=\mathbb{P}^1\times \mathbb{P}^1\setminus \mathbb{A}^2$ and smooth outside this place. The result then follows from the fact that the number of places at the infinity of $\mathbb{P}^2$ is equal to the number of places on the boundary  $B$ of $\mathbb{P}^1\times \mathbb{P}^1$.

$(c)\Leftrightarrow (d)$ This is an immediate consequence of the embedding line theorem  \cite{AM} (see also \cite{S2}).
\end{proof}

\subsection{The monic reducible case. Proof of Theorem \ref{tminimal}.}

\begin{proposition}\label{prop:resultant1}
Let $g, h\in \mathbb{K}[x,y]$ be two monic minimal polynomials. Then 
\[
\Res_y(g,h)\in \mathbb{K}^* \iff h=\mu g+\lambda
\]
for some nonzero constants $\mu,\lambda\in \mathbb{K}^*$. 
\end{proposition}

\begin{proof}
We have $\Res_y(g,h)\in \mathbb{K}^*$ if and only if the curves $C_1, C_2\subset \mathbb{P}^1\times\mathbb{P}^1$ respectively defined by $g$ and $h$ do not intersect in the open set $\mathbb{A}^1\times \mathbb{P}^1$. Let $\sigma\in \Aut(\mathbb{A}^2)$ and let $\tilde{g}=g\circ \sigma$ and $\tilde{h}=h\circ\sigma$. Assume that $\deg_y \tilde{g} >0$. Since $g$ is assumed to be monic minimal, so is $\tilde{g}$ by Theorem \ref{t3}. It follows that the respective curves $\tilde{C}_1$ and $\tilde{C}_2$ of $\tilde{g} $ and $\tilde{h}$ do not intersect in  $\mathbb{A}^1\times \{\infty\}$. Since $C_1$ and $C_2$ do not intersect in $\mathbb{A}^2$ by assumption, the curves $\tilde{C}_1$ and $\tilde{C}_2$ can not intersect in $\mathbb{A}^2$ since $\sigma$ is an automorphism of the plane. Hence $\tilde{C}_1$ and $\tilde{C}_2$ do not intersect in $\mathbb{A}^1\times \mathbb{P}^1$, that is 
\begin{equation}\label{resconstant}
\Res_y(\tilde{g} ,\tilde{h})\in \mathbb{K}^*.
\end{equation}
By Theorem \ref{t3}, there exists $\sigma\in \Aut(\mathbb{A}^2)$ such that $\tilde{g} =y$. Combined with (\ref{resconstant}), this implies that $\tilde{h}(x,0)\in \mathbb{K}^*$. Since $\tilde{h} $ is a coordinate polynomial by Theorem \ref{t3}, Lemma \ref{lemma:stillirreducible} (Appendix \ref{C}) implies that $\tilde{h}(x,y)-\tilde{h}(x,0)$ is irreducible, forcing the equality  $deg_y \tilde{h}=1$. Since $h$ is monic, so is $\tilde{h}$ and the condition $\tilde{h}(x,0)\in \mathbb{K}^*$ implies that $\tilde{h}=\mu y + \lambda=\mu \tilde{g} +\lambda$ for some constant $\mu,\lambda\in \mathbb{K}^*$. The result follows by applying $\sigma^{-1}$. 
\end{proof}

\noindent
\paragraph*{Proof of Theorem \ref{tminimal}.} Let $f$ be a monic separable polynomial with $r$ irreducible factors $f_1,\ldots,f_r$ of respective degrees $d_1,\ldots,d_r$. Corollary \ref{cmin} combined with the multiplicative properties of the discriminant gives the inequality 
\[
\deg_y(\Delta_y(f))= \sum_{i=1}^r \deg_y(\Delta_y(f_i)) + \sum_{i\ne j}^r \deg_y(\Res_y(f_i,f_j)) \ge \sum_{i=1}^r (d_i-1)\ge d_y-r.
\]
Moreover, equality holds if and only if all factors $f_i$ are minimal and satisfy $\Res_y(f_i,f_j)\in \mathbb{K}^*$ for all $i\ne j$. If $f$ is monic, all its factors are also monic. We conclude thanks to Theorem \ref{t3} and Proposition \ref{prop:resultant1} that there exists an automorphism $\sigma\in \Aut(\mathbb{K}^2)$ such that $f\circ \sigma$ is a degree $r$ univariate polynomial. 
Note that $r$ automatically divides $d_y$. $\hfill\square$


\section{$GL_2(\mathbb{K}[x])$-orbits of minimal polynomials}\label{SSnonmonic}

We saw that monic minimal polynomials are particularly easy to describe and construct since they coincide with coordinate polynomials. What can be said for nonmonic minimal polynomials ? Thanks to the relation (\ref{invariance}), an easy way to produce nonmonic minimal polynomials is to let act $G:=GL_2(\mathbb{K}[x])$ on a monic minimal polynomial. It is natural to ask if all nonmonic minimal polynomials arise in such a way. We prove here that the answer is no, a counterexample being given by $f= x(x-y^2)^2-2\lambda y(x-y^2)+\lambda^2$ (Theorem \ref{contrexintro} of the introduction). However, we will show that if we assume that $d_y$ is prime, then the answer is yes (Theorem \ref{nonmonicintro}). Both results will follow from divisibility constraints on the partial degrees of a minimal polynomial (Theorem \ref{reduced} and Theorem \ref{conj}).

\begin{definition}
Let $f, g\in \mathbb{K}[x,y]$ be two irreducible polynomials with partial degrees $\deg_y f>1$ and $\deg_y g>1$. We say that $f$ and $g$ are $G$--equivalent, denoted by $f\equiv g$, if there exists $\sigma\in G$ such that $f=\sigma(g)$, the action of $\sigma$ being defined in (\ref{action}). 
\end{definition}

\subsection{$G$-reduction of minimal polynomials. Proof of Theorem \ref{contrexintro}.}\label{orbits}

In this subsection, we focus on the $G$--reduction of minimal polynomials: what is the 'simplest' form of a polynomial in the $G$--orbit of a minimal one ?

\noindent
\paragraph*{Newton polytope.}
We define the \textit{generic} Newton polytope of $f\in \mathbb{K}[x,y]$ as the convex hull
\[
P(f):=\Conv\Big( (0,0)\cup (0,d_y) \cup \Supp(f)\Big),
\]
where $\Supp(f)$ stands for the support of $f$, {\it i.e} the set of exponents that appear in its monomial expansion. It is well known that the edges of the generic polytope that do not pass throw the origin give information about the singularities of $f$ at infinity. In our context, we have the following lemma:

\begin{lemma}\label{pol}
Suppose that $f$ is minimal. Then 
\[
P(f):=\Conv\Big( (0,0),(0,d_y),(b,d_y),(0,a)\Big)
\] 
for some integers $a, b$. 
\end{lemma}

\begin{proof}
Since $f$ has a unique place along $x=\infty$, the claim follows from Newton-Puiseux Factorisation Theorem applied along the line $x=\infty$. See for instance {\cite[Chap.\,6]{Cas}}.
\end{proof}

The integers $a=a(f)$ and $b=b(f)$ of Lemma \ref{pol} coincide with 
the respective degrees in $x$ of the constant and leading coefficients of $f$ with respect to $y$. 
Thanks to the previous lemma, we have the relation 
\[
d_x=\max(a,b)
\]
for any  minimal polynomial $f$ and we define the integer $c=c(f)$ as
\[
c:=\min (a,b).
\]
We say that $f$ is in \textit{normal position} if $b\le a$, that is if  $(d_x,c)=(a,b)$. 

\begin{figure}[h]
    \begin{center}
      \resizebox{9.5cm}{!}{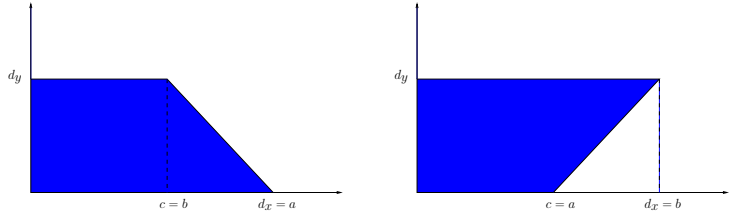}
    \end{center}  
    \caption{The generic Newton polytopes of a minimal polynomial in normal and non normal position.} 
\end{figure}

\paragraph*{Reduced minimal polynomials.}
We can enounce now our main result about $G$--reduction of minimal polynomials. Given $n:\mathbb{K}[x,y]\to \mathbb{Q}^+$ and $f\in \mathbb{K}[x,y]$, we define 
\[
n_{\min} (f):=\inf\{n(g),\,\,g\equiv f\}.
\]

\begin{theorem}\label{reduced}
Let $f$ be a minimal polynomial with parameters $(d_y,d_x,c)$. Denote by $V$ the euclidean volume of $P(f)$. Suppose that $d_y\ge 2$. Then the following assertions are equivalent:
\begin{enumerate}
\item $V=V_{\min}$
\item $d_x=d_{x,\min}$ and $c=c_{\min}$
\item $d_y$ does not divide $d_x-c$.
\end{enumerate}
\end{theorem}

\begin{definition}
We say that $f$ is \textit{reduced} if it is minimal and satisfies one of the equivalent conditions of Theorem \ref{reduced}. 
\end{definition}

The remaining part of this subsection is dedicated to the proof of Theorem \ref{reduced}.

\paragraph*{The characteristic polynomial.} It turns out that Newton-Puiseux Theorem gives strong information about the edge polynomial of $f$ attached to the right hand side of $P(f)$. Namely, we have:

\begin{lemma}\label{facette}
Suppose that $f$ is minimal with parameters $(d_y,a,b)$. Then  \begin{equation}\label{eqfacette}
\displaystyle f(x,y)=x^b(\alpha y^p+\beta x^q)^n+\sum_{\substack{j\le pn \\ pi+qj<pqn+np}} c_{ij} x^i y^j
\end{equation}
where $p\in \mathbb{N}^*$ and $q\in \mathbb{Z}$ are coprime integers such that
\begin{equation}\label{eqab}
pn=d_y,\quad qn=a-b,
\end{equation}
where $\alpha,\beta \in \mathbb{K}^*$. We call the polynomial $f_{\infty}:=(\alpha y^p+\beta x^q)^n$ the characteristic polynomial of $f$ at $x=\infty$. 
\end{lemma}

\begin{proof}
By Corollary \ref{cmin}, the Zariski closure in $\mathbb{P}^1\times \mathbb{P}^1$ of the curve defined by $f$ has a unique place along the line $x=\infty$. Thus, it  follows once again from the Newton-Puiseux Theorem applied along the line $x=\infty$ that the edge polynomial attached to the right hand edge of $P(f)$ is of the form $x^b g(x,y)$ where $g$ is the power of an irreducible quasi-homogeneous polynomial {\cite[Chap.\,6]{Cas}}. 
\end{proof}

\begin{corollary}\label{c13}
If $V=V_{\min}$, then $d_y$ does not divide $d_x-c$.
\end{corollary}

\begin{proof}
By Lemma \ref{tau} below, the parameters $(d_x,d_y,c,V)$ are invariant under the inversion $\tau$ while the parameters $(a,b)$ are permuted. Hence there is no loss to suppose that $f$ is in normal position, that is $(d_x,c)=(a,b)$. By (\ref{eqab}) in Lemma \ref{facette}, we get that $q\ge 0$ and that $d_y$ divides $d_x-c$ if and only if $p=1$. In such a case, the polynomial
\[
g(x,y):=f(x,y-\beta/\alpha x^q) 
\]
satisfies $b(g)=b(f)$ and $a(g)< a(f)$. Since $g$ is equivalent to $f$, it is also minimal of partial degree $d_y$, and we deduce from Lemma \ref{pol} that
\[
V(g)=\frac{d_y(a(g)+b(g))}{2} < V(f)=\frac{d_y(a(f)+b(f))}{2}.
\]
The corollary follows. 
\end{proof}

\noindent
\paragraph*{Basic transformations.} Let us first study the behaviour of the parameters $d_x$ and $c$ under the inversion and the polynomial De Jonqui\`eres transformations.
We define the inversion $\tau\in G$ by
\[
\tau(f):=y^{d_y} f(x,1/y).
\]
We have the following obvious lemma:

\begin{lemma}\label{tau}
Let $f\in \mathbb{K}[x,y]$ not divisible by $y$. The parameters $d_y,d_x,c$ are invariant by $\tau$ and the parameters $a$ and $b$ are permuted. 
\end{lemma}

\begin{proof} It is straightforward to check that $d_y(g)<d_y(f)$ if and only if $f(x,y)=y^k h(x,y)$ with $k>0$, which is excluded by hypothesis. The remaining part of the lemma is straightforward.
\end{proof}

Let $U\subset G$ stands for the polynomial De Jonqui\`eres subgroup of $G$, that is the subgroup of transformations $\sigma$ of type
\[
\sigma(f):(x,y)\longmapsto f(x,\lambda y + h(x)),
\] 
where $\lambda\in \mathbb{K}^*$ and $h\in \mathbb{K}[x]$. We define then $\deg(\sigma):=\deg(h)$, with the convention $\deg(0)=0$. If $\sigma$ is an homothety, that is if $h=0$, then the Newton polytope and all the parameters of $f$ and $\sigma(f)$ obviously coincide. Otherwise, we get:

\begin{lemma}\label{sigma}
Let $f$ be a minimal polynomial of degree $d_x >0$ and let $\sigma\in U$ not an homothety. Let $g=\sigma(f)$. Then:
\begin{enumerate}
\item If $f$ is in normal position and $d_y$ does not divide $d_x-c$ then 
\[
d_x(g)=\max(c(f)+d_y(f)\deg \sigma,d_x(f))\quad {\rm and} \quad c(g)=c(f)
\]
\item If $f$ is not in normal position then 
\[
d_x(g)=d_x(f)+d_y(f)\deg \sigma\quad {\rm and} \quad c(g)=d_x(f).
\]
\end{enumerate}
In both cases, $g$ is in normal position.
\end{lemma}

\begin{proof}
Let us write $\sigma(f)=f(x,\lambda y+\mu x^k+r(x))$, with $\lambda, \mu\in \mathbb{K}^*$,  $k=\deg \sigma \ge 0$  and $\deg r <k$ and let us write $f=\sum c_{ij} x^{i} y^{j}$. We have
\begin{equation}\label{i+2}
g(x,0)=f(x,\mu x^k + r(x))=\sum_{i+kj=M} c_{ij} \mu^j x^{i+kj}+R(x)
\end{equation}
where 
\[
M:=\max_{(i,j)\in Supp(f)} (i+kj)\quad {\rm and}\quad \deg R < M,
\]
Since the line $i+k j=0$ has negative slope $-1/k$ (vertical if $k=0$), Lemma \ref{pol} forces $M$ to be reached at one of the two vertices $(a(f),0)$ or $(b(f),d_y)$ of $N_f$, forcing the equality
\[
M=\max (a(f),b(f)+k d_y).
\]
Suppose that $f$ is not in normal position. Then $a(f)<b(f)$ and $M=b(f)+k d_y$ is reached at the unique point $(b(f),d_y)$ of $N_f$. Thus, there is a unique monomial in $(\ref{i+2})$ with maximal degree. This forces the equality 
\[
a(g):=\deg_x(g(x,0))=b(f)+k d_y,
\]
On another hand it is clear that for any $f$, we have
\[
\lc_y(g)=\lambda^{d_y} \lc_y(f).
\]
In particular, $b(g)=b(f)=a(g)-k d_y\le a(g)$ forcing equality $(a(g),b(g))=(d_x(g),c(g))$. Since $f$ is not in normal position, we have $(a(f),b(f))=(c(f),d_x(f))$. Claim $(2)$ follows.
Suppose now that $f$ is in normal position and that $d_y$ does not divide $d_x-c$. In particular, we have $a(f)\ne b(f)+k d_y$ so that once again $M$ is reached at a unique point of $N_f$, forcing equality 
\[
a(g):=\deg_x(g(x,0))=\max(b(f)+k d_y,a(f)).
\]
Since $b(g)=b(f)\le a(g)$ we have $(a(g),b(g))=(d_x(g),c(g))$. Since $f$ is in normal position, we have $(a(f),b(f))=(d_x(f),c(f))$. Claim $(1)$ follows. 
\end{proof}

\paragraph*{Decomposition of $GL_2(\mathbb{K}[x])$.} 
Let $V:=GL_2(\mathbb{K})\subset G$. It is well known that $GL_2(\mathbb{K}[x])$ is the amalgamate product of the subgroups $U$ and $V$ along their intersections, see \cite{Nagao} for instance. On another hand, it is a classical fact that $V$ is  generated by translations $y\to y+\lambda$, homotheties $y\to \lambda y$, $\lambda \in \mathbb{K}^*$ and the inversion $\tau$. Since translations and homotheties lie in $U\cap V$, it follows that any transformation $\sigma\in G$ can be decomposed as an alternate product 
\begin{equation}\label{dec}
\sigma=\sigma_n\tau \sigma_{n-1} \tau \cdots \sigma_2\tau\sigma_1, 
\end{equation}
with  $\sigma_i\in U$ for all $i$. We can assume moreover that $\sigma_i\notin U\cap V$ except possibly for $i=1$ or $i=n$, that is
\[
\deg \sigma_i > 0 \qquad \forall\,\, i=2,\ldots,n-1.
\]
Let now $\sigma\in G$ having decomposition (\ref{dec}) and let $f\in \mathbb{K}[x,y]$.  We introduce the notation
\[
f_1=\sigma_1(f)\quad {\rm and}\quad f_i=(\sigma_i \tau)(f_{i-1}), \,\, i=2,\ldots,n.
\] 
and we write for short $d_i=d_x(f_i)$ and  $c_i=c(f_i)$. The following proposition has to be compared to \cite{wight} where the author considers the behaviour of the total degree of a bivariate polynomial under the action of $\Aut(\mathbb{A}^2)$.

\begin{proposition}\label{sequence}
Let $f$ be a minimal polynomial in normal position such that $d_y$ does not divide $d_x-c$ and let $\sigma\in G$. With the notation introduced before, we have
\[
d_x\le d_1< d_2< \cdots < d_{n-1}\le d_n \quad {\rm and } \quad c=c_1< c_2< \cdots < c_{n-1}\le c_n.
\]
Moreover $d_y$ does not divides $d_n-c_n$ if and only if $(d_n,c_n)= (d_x,c)$.
\end{proposition}

\begin{proof}
By Lemma \ref{sigma}, the proposition is true if $n=1$. We have $a(f)>b(f)$ by assumption so that $a(f_1)>b(f_1)$ by Lemma \ref{sigma}. By Lemma \ref{tau}, it follows that $\tau(f_1)$ is not in normal position. If $\sigma_2$ is an homothety, then $n=2$ and $f_2=\sigma(\tau(f_1))$ has the same parameters than $f_1$, proving the proposition in that case. If $\sigma_2$ is not an homothety, then $d_2 > d_1$, $c_2>c_1$ and $d_y$ divides $d_2-c_2$ by Lemma \ref{sigma}. Hence the Proposition follows for $n=2$. Moreover we have 
\[
n>2\quad \Longrightarrow \quad \deg \sigma_2>0 \quad \Longrightarrow\quad  a(f_2)>b(f_2)
\]
the second implication using again Lemma \ref{sigma}. Thus $n>2$ implies moreover that $\tau(f_2)$ is not in normal position. The Proposition then follows by induction. 
\end{proof}

\paragraph*{Proof of Theorem \ref{reduced}.} We have $(1)\Rightarrow (3)$ by Corollary \ref{c13} while the implication $(3)\Rightarrow (2)$ is an immediate consequence of Proposition \ref{sequence}. The remaining implication $(2)\Rightarrow (1)$ follows from equality $V=d_y(c+d_x)/2$ that holds for minimal polynomials thanks to Lemma \ref{pol}. $\hfill\square$

\vskip2mm
\noindent

As announced at the beginning of the section, Theorem \ref{contrexintro} is an easy corollary of Theorem \ref{reduced}.

\paragraph*{Proof of Theorem \ref{contrexintro}.} A direct computation shows that the polynomial $f= x(x-y^2)^2-2\lambda y(x-y^2)+\lambda^2$ is minimal with parameters $(d_y,d_x,c)=(4,3,1)$ for all $\lambda\in \mathbb{K}^*$ (for $\lambda=0$, the polynomial $f$ is reducible). Since $d_y$ does not divide $d_x-c$, $f$ is reduced by  Theorem \ref{reduced}. Hence $c=c_{\min}=1\ne 0$ and $f$ is not equivalent to a monic polynomial by Theorem \ref{reduced}. $\hfill{\square}$

\subsection{Cremona equivalence of minimal polynomials.}\label{Cremona} 

Theorem \ref{contrexintro} shows that we can not hope that a nonmonic minimal polynomial can be transformed to a coordinate via a composition of an element of $GL_2(\mathbb{K}[x])$ with an element of $\Aut(\mathbb{A}^2)$. However, both groups act on the curve of $f$ as subgroups of the Cremona group $\Bir(\mathbb{A}^2)$ of birational transformations of the plane, and both Theorem \ref{t3} and $GL_2(\mathbb{K}[x])$-invariance of the degree of the discriminant leads us to ask the natural following question:

\begin{question}\label{q1}
Do minimal polynomials define curves Cremona equivalent to lines ?
\end{question}

Theorem \ref{t3} gives a positive answer in the case of monic polynomials, and more generally for all members of their $GL_2(\mathbb{K}[x])$-orbits. This is also the case for the nonmonic minimal polynomial of Theorem \ref{contrexintro} as it will be shown in the next Proposition. Note that being Cremona equivalent to a line does not imply minimality.
In a close context, it has recently been proved in \cite{KP} that any rational cuspidal curve of $\mathbb{P}^2$ is Cremona equivalent  to a line, solving a famous problem of Coolidge and Nagata. In the present context, minimal polynomials define rational unicuspidal curves of $\mathbb{P}^1\times\mathbb{P}^1$ (Corollary \ref{cmin}) and we may ask whether the result of Koras-Palka extends to this case.
These kind of problems are closely related to the geometry of the minimal embedded resolution.

\begin{proposition}\label{prop:exemple}
The curve defined be the polynomial $f= x(x-y^2)^2-2\lambda y(x-y^2)+\lambda^2$ is Cremona equivalent to a line.
\end{proposition}

\begin{proof}
The polynomial $f$ being minimal with parameters $(d_x,d_y,c)=(3,4,1)$, it is easy to see that it defines a unicuspidal curve of $\mathbb{P}^2$. Hence the claim follows from \cite{KP}. It has to be noticed that we can 'read' the underlying birational transformation on the Newton polytope of $f$. We have $f_1(x,y):=f(x+y^2,y)=x^3 +(xy-\lambda)^2$ and $f_2(x,y):=f_1(x,y/x+\lambda)=x^3+y^2$ defines a curve which is clearly Cremona equivalent to $f=0$. Let $C\subset \mathbb{P}^2$ be the projective plane curve defined by the homogenisation $F(X,Y,Z)=X^3+Y^2Z$ of $f_2$. Consider the rational map 
\begin{eqnarray*}
\mathbb{P}^2 & \dashrightarrow &  \mathbb{P}^2 \\ (X:Y:Z)  & \mapsto & (XY^2:Y^3:X^3+Y^2 Z).
\end{eqnarray*}
The restriction of $\sigma$ to the chart $Y=1$ coincides with the affine map $(x,z)\to (x,x^3+z)$ which is clearly invertible. Hence $\sigma\in \Bir(\mathbb{P}^2)$ is a Cremona transformation that satisfies $\sigma^{-1}(Y=0)=C$. 
\end{proof}

\subsection{Divisibility constraints for minimal reduced polynomials. Proof of Theorem \ref{nonmonicintro}.}\label{toric}


Thanks to Theorem \ref{tminimal}, monic minimal polynomials coincide with coordinate polynomials. In particular, it follows from \cite{Ab} that they obey to the crucial property:

\begin{proposition}[(Abhyankar-Moh's Theorem reformulated)]\label{prop:conj}
Let $f$ be a monic minimal polynomial. Then  $d_x$ divides $d_y$ or $d_y$ divides $d_x$. 
\end{proposition}

Proposition \ref{prop:conj} is another reformulation of the embedding line theorem of Abhyankar-Moh \cite{Ab}. Indeed, this property allows to reduce the degree of $f$ with translations $x\mapsto x-\alpha y^k$ or $y\mapsto y-\alpha x^k$. Since these translations preserve the property of being simultaneously monic and minimal, we can reach $f=y$. In the nonmonic case, a similar reduction process requires a positive answer to the following question: 

\begin{question}\label{q2}
If $f$ is minimal, is it true that $d_x-c$ divides $d_y$ or $d_y$ divides $d_x-c$ ?
\end{question}

Here, the parameter $c$ is the one defined in the previous Subsection \ref{orbits}. 
This property holds for all polynomials in the $G$-orbit of a monic minimal polynomial by Proposition \ref{prop:conj} and Proposition \ref{sequence}. It also holds for the minimal reduced polynomial $f$ of Theorem \ref{contrexintro} ($d_x-c=2$ divides $d_y=4$), and might be seen as a key point in the explicit construction of the birational map of  Proposition \ref{prop:exemple}.  Although Questions \ref{q1} and \ref{q2} are closely related, translations on $x$ do not preserve the minimality of a nonmonic minimal polynomial, and it is not clear that a positive answer to Question \ref{q2} leads to a positive answer to Question \ref{q1}. Anyway, it would be an important property for reducing minimal polynomials to a "nice canonical form".  We prove here a partial result that shows that if $f$ is minimal and $d_y$ does not divide $d_x-c$ then  $d_y$ and $d_x-c$ are not coprime as soon as $d_x>1$ . 

\begin{theorem}\label{conj}
Let $f$ be a minimal polynomial of degree $d_y\ge 1$. If $f$ is nonreduced, then $d_y$ divides $d_x-c$. If $f$ is reduced, we have:
\begin{enumerate}
\item If $d_x=0$ then $c=0$ and $d_y=1$.
\item If $d_x=1$ then $c=0$ and $d_y>1$.
\item If $d_x>1$ and $c=0$ then $d_x$ divides $d_y$. 
\item If $d_x>1$ and $c>0$ then $2\le \gcd(d_x-c,d_y)\le d_y/2$.
\end{enumerate}
\end{theorem}

\paragraph*{Proof of Theorem \ref{conj}.} If $f$ is nonreduced, then $d_y$ divides $d_x-c$ by  Theorem \ref{reduced}. Assume that $f$ is reduced. If $d_x=0$, then $c=0$ is obvious and $d_y=1$ since otherwise $f$ would not be irreducible. If $d_x=1$, then $c\le 1$. Since $f$ is reduced, we must have $c=0$ and $d_y>1$ by Theorem \ref{reduced} since otherwise $d_y$ would divide $d_x-c$. Suppose now that $d_x>1$. If $c=0$, then we can suppose that $f$ is monic up to apply the inversion $y\to 1/y$. The claim thus follows from  Proposition \ref{prop:conj} combined with the fact that $d_y$ can not divide $d_x$ since $f$ is assumed to be reduced (Theorem \ref{reduced}). Suppose now that $d_x>1$ and $c>0$. Then $\gcd(d_x-c,d_y)\le d_y/2$ by Theorem \ref{reduced}. Let $P:=P(f)$ be the generic Newton polytope of $f$. Let $X$ be the complete simplicial toric surface associated to the normal fan of $P$ and let $C\subset X$ be the curve defined by $f$. Since $c>0$, $P$ has exactly four edges.   To the right hand edge $\Lambda$ of $P$ corresponds a toric divisor $E\subset X$ such that $E\simeq \mathbb{P}^1$,
\[
X\setminus E=\mathbb{A}^1\times \mathbb{P}^1\quad {\rm and}\quad E\cdot C=\Card(\Lambda \cap \mathbb{Z}^2)-1,
\]
where $E\cdot C$ stands for the intersection degree. In particular we have by minimality of $f$ that $C$ is smooth in $X\setminus E$. 
Now, we have $\gcd(d_x-c,d_y)=1$ if and only if $\Lambda$ has no interior lattice points, that is if and only if $C\cdot E=1$. Since both $C$ and $E$ are effective divisors, it follows in particular that $C$ intersects $E$ at a unique point and is transversal to $E$ at that point. In particular it is smooth along $E$, hence smooth in $X$ by what we said before. The genus formula for smooth curves in toric surface, combined with the rationality of $C$ leads to the equality
\[
0=g(C)=\Card( \interieur(P)\cap \mathbb{Z}^2),
\]
where $\interieur(P)$ stands for the interior of $P$. But this contradicts the fact that $P$ is the convex hull of $(0,0),(0,d_y),(c,d_y),(d_x,0)$ with $d_x>1$, $c>0$ and $d_y\ge 2$.  $\hfill{\square}$

\paragraph*{Proof of Theorem \ref{nonmonicintro}.} It is an immediate consequence of Theorem \ref{conj}. Namely, if $d_y$ is prime, only case $(2)$ in Theorem \ref{conj} can occur for a reduced form of $f$. $\hfill{\square}$

\paragraph*{Another proof of Theorem \ref{nonmonicintro}.} We found it instructive to give a direct proof of Theorem \ref{nonmonicintro} that only uses properties of the discriminant. Let $f$ be a minimal polynomial with parameters $a,b$ and $d=d_y$ prime. Let us write $d=pn$ and $a-b=qn$ with $p,q$ coprime, as in equation (\ref{facette}). We must have $d_{x,\min}>0$ since otherwise $f$ would have a constant discriminant, contradicting minimality and $d\ge 2$. It follows from Corollary \ref{c13} that a suitable $GL_2(\mathbb{K}[x])$-reduction leads to the case $p\ne 1$. One can suppose also that $a-b>0$ up to apply an inversion. Since $d$ is assumed to be prime, it follows that $p=d$ and $n=1$. 

By equation (\ref{facette}), the Newton polytope of $f$ has a unique edge $\Lambda$ that connects the points $(d,b)$ and $(0,a)$, with slope $\alpha:=\frac{d}{b-a}$, that is
\[
\Lambda=\big\{(i,j)\in \mathbb{N}^2\,\,|\,\, i(a-b)+jd=ad,\,\,0\le j \le d\big\}.
\]
Since $d$ is coprime to $a-b$, all lattice points of the polytope of $f$ lie below $\Lambda$, except $(0,a)$ and $(d,b)$. Let $f=\sum_{j=0}^{d} f_j y^j$ with $f_j\in \mathbb{K}[x]$, and let $n_j:=\deg_x(f_j)$. We get that 
\begin{equation}\label{ineq_facette}
n_0=a,\quad n_d=b,\quad {\rm and} \quad d n_j+(a-b)j<ad  \quad \forall\,\, j\ne 0,d.
\end{equation}
By Lemma \ref{lemma:genericdisc_p0}, we have
\[
\Delta_y(f) = (-1)^{d(d-1)/2} d^d f_d^{d-1} f_0^{d-1} + o(f_0^{d-1}) \; \]
as a polynomial in $f_0$. Let $c f_0^{\beta_0}\cdots f_d^{\beta_d}$ be a monomial appearing in $\Delta_y(f)$, with $c\in \mathbb{K}^*$. It is a well known fact that $\Delta_y(f)$ is a homogeneous polynomial in $(f_0,\ldots,f_d)$ of degree $2(d-1)$, and a quasi-homogeneous polynomial of weighted degree $d(d-1)$ with respect to the weight $(0,1,\ldots,d)$.   In other words, we have:
\begin{equation}\label{homog}
\sum_{j=0}^d \beta_j =2(d-1)\quad {\rm and}\quad \sum_{j=0}^d j\beta_j=d(d-1).
\end{equation}
We have $\deg_x(f_d^{d-1} f_0^{d-1})=(a+b)(d-1)$ while $\deg_x(\Delta_y(f))=d-1$ by minimality of $f$. Suppose that $a+b > 1$. Then there must appear at least another monomial  $c f_0^{\beta_0}\cdots f_d^{\beta_d}$ in $\Delta_y(f)$ with a nonzero coefficient $c\in \mathbb{K}^*$ and such that
\[
\deg_x(f_0^{\beta_0}\cdots f_d^{\beta_d})=(a+b)(d-1).
\]
By (\ref{homog}), and since $d>1$, we see that $f_d^{d-1} f_0^{d-1}$ is the only monomial in $\Delta_y(f)$ that involves only $f_0$ and $f_d$. Hence there exists at least one exponent $\beta_j>0$ for some $0<j<d$. Since $a-b>0$ by assumption, (\ref{ineq_facette}) and (\ref{homog}) lead to a strict inequality
\begin{eqnarray*}
\qquad\qquad \deg_x(f_0^{\beta_0}\cdots f_d^{\beta_d})=\sum_{j=0}^d \beta_j n_j & < & \sum_{j=0}^d \beta_j \Big(a-\frac{j(a-b)}{d}\Big)\\
& = & a\sum_{j=0}^d \beta_j+\frac{(a-b)}{d}\sum_{j=0}^d j\beta_j\\
& = & 2a(d-1)-(a-b)(d-1) \\
& =  & (a+b)(d-1),
\end{eqnarray*}
leading to a contradiction. Hence $a+b=1$, that is $a=1$ and $b=0$ since we assumed $a-b>0$. It follows that $f(x,y)=g(y)+\lambda x$ for some $g\in \mathbb{K}[y]$ of degree $d$ and for some $\lambda \in \mathbb{K}^*$. $\hfill\square$

\section{A uniform lower bound for reducible polynomials}\label{Sfin}

We now focus on the non monic reducible case and we prove Theorem \ref{treducible} of the introduction: all polynomials $f\in \mathbb{K}[x,y]$ with non constant discriminant satisfy
\[
\deg_x \Delta_y(f) \ge \Big\lceil \frac{d_y-1}{2}\Big\rceil
\]
and we have a complete classification of polynomials for which equality holds.
The proof requires some preliminary lemmas. In order to study the discriminant of reducible polynomials, it is more convenient 
to consider homogeneous polynomials in $Y=(Y_0:Y_1)$. The homogeneity in $x$
is not necessary. We thus consider polynomials $F\in \mathbb{K}[x][Y]$. 

\begin{lemma}\label{lemma:reducible1}
Let $F \in \mathbb{K}[x][Y]$ be a squarefree polynomial of degree $\deg_Y F = d \ge 0$
with no factor in $\mathbb{K}[x]$.
Assume that $F$ has only linear factors.
Then exactly one of the following occurs:
\begin{enumerate}
\item\label{case1} $\deg_x \Delta_Y F = 0$ and $F$ is $G$--equivalent to some polynomial of $\mathbb{K}[Y]$.
\item\label{case2} $d = 2$   and $\deg_x \Delta_Y F \ge 2      > \frac d2 $.
\item\label{case3} $d \ge 3$ and $\deg_x \Delta_Y F \ge 2(d-2) > \frac d2$.
\end{enumerate}
\end{lemma}
\begin{proof}
The cases $d = 0$ and $d = 1$ are trivially in case (\ref{case1}).
We now assume that $d \ge 2$.
We have $F = \prod_{i=1}^d F_i$ with $F_i = a_i Y_0 + b_i Y_1$, for
some $a_i$ and $b_i$ in $\mathbb{K}[x]$.
For all nonempty subset $I \subset \{1,\dots,d\}$, we write $F_I = \prod_{i\in I} F_i$. If $I$ has only $1$ element, then clearly $\Delta_Y F_I \in \mathbb{K}$.
Among all subsets $I$ such that $\Delta_Y F_I \in \mathbb{K}$, we consider
one with a maximal number of elements and write $m$ for its cardinality.
We have $1 \le m \le d$.

Consider first the case $m = 1$.
For all $i\neq j$, we have $\deg_x \Res_Y(F_i,F_j) \ge 1$.
This implies that $\deg_x \Delta_Y F \ge d(d-1)$.
This proves the case (\ref{case2}) if $d=2$ and the case (\ref{case3}) if $d>2$.

Consider now the case $2 \le m$. We can assume that $I = \{1,2,\dots,m\}$.
We have then $\Res(F_1,F_2) \in \mathbb{K}$. The matrix
$\sigma = \begin{pmatrix} b_2 & -b_1 \\ -a_2 & a_1 \end{pmatrix}$
is therefore an element of $GL_2(\mathbb{K}[x])$.
Via the action of $\sigma$, $F_1$ and $F_2$ are transformed into
$Y_0$ and $Y_1$. Without loss of generality, we assume that
$F_1 = Y_0$ and $F_2 = Y_1$. For all $3 \le i \le m$, we have
$\Res_Y(F_i,Y_0) \in \mathbb{K}$, hence $b_i \in \mathbb{K}$. Similarly, we have
$\Res_Y(F_i,Y_1) \in \mathbb{K}$, hence $a_i \in \mathbb{K}$. This proves that
$F_i \in \mathbb{K}[Y]$ for all $i \in I$.
If $m=d$, then we have proved that $F$ is equivalent to a polynomial in $\mathbb{K}[Y]$, hence we are in case (\ref{case1}). 

It remains to consider the case $2\le m<d$. This case is possible only if
$d \ge 3$.
As before, we can assume that $I = \{1,2,\dots,m\}$ and $F_i \in \mathbb{K}[Y]$ for
all $i \in I$.
For an integer $j \not\in I$, there exists at most one value of $i \in I$
such that $\Res_Y(F_i,F_j) \in \mathbb{K}$. Otherwise, using
a similar argument as before, we would have $F_j \in \mathbb{K}[Y]$ and
$\Res_Y(F_i,F_j) \in \mathbb{K}$ for all $i \in I$, contradicting the maximality of $I$.
Since each $F_j$ for $j\not\in I$ has at least $m-1$ nonconstant
resultants with $F_i$ for $i \in I$, this proves that
$\deg_x \Delta_Y F \ge 2(m-1)(d-m)$. It is an exercise to verify
the inequalities $2(m-1)(d-m) \ge 2(d-2) > \frac d2$.
\end{proof}

\begin{lemma}\label{lemma:minimalresultant1}
Let $F \in \mathbb{K}[x][Y]$ be an irreducible polynomial of degree $d \ge 2$.
Assume that $F$ is minimal. Consider an integer $n$ and polynomials
$F_i = a_iY_0 + b_iY_1$, for $1 \le i \le n$ and $a_i,b_i \in \mathbb{K}$, that are
pairwise coprime. 
If $\Res_Y(F,F_i) \in \mathbb{K}$ for all $1\le i \le n$, then $n \le 1$.
\end{lemma}
\begin{proof}
It is enough to prove that the case $n=2$ is impossible.
Suppose that two such polynomials exist.
Using the action of $GL_2(\mathbb{K})$, we can assume that $F_1 = Y_0$ and
$F_2 = Y_1$. We write $r_1 = \Res_Y(F,Y_1) \in \mathbb{K}$.
The relation $\Res_Y(F,Y_0) \in \mathbb{K}$ implies that
$F(Y_0,Y_1)$ is monic in $Y_1$. By Theorem \ref{t3}, $F(1,y)$ is equivalent
to $y$ up to an automorphism of $\mathbb{A}^2$, so that we can apply Lemma 
\ref{lemma:stillirreducible} and deduce that $F(1,y)-r_1$ is irreducible
of degree $d>2$. However, it is by construction divisible by $y$.
We get a contradiction.
\end{proof}

\begin{lemma}\label{lemma:reducible2}
Let $F \in \mathbb{K}[x][Y]$ be a squarefree polynomial of degree $\deg_Y F = d \ge 2$
with no factor in $\mathbb{K}[x]$.
Assume that $F=PQ$, where $P$ is irreducible of degree $\deg_Y P \ge 2$,
and $Q$ has only linear factors.
Then
\[ \deg_x \Delta_Y F \ge \big\lceil \frac {d-1}2 \big\rceil \]
Furthermore, equality holds if and only if $F$ is $G$--equivalent to one
of the following exceptional polynomials:
\begin{itemize}
\item (case $d=2$): $Y_0^2+(x+a)Y_1^2$, ($a \in \mathbb{K}$)
\item (case $d=3$): $Y_1(Y_0^2+(x+a)Y_1^2)$, ($a \in \mathbb{K}$)
\item (case $d=4$): $Y_1(Y_0^3+aY_0Y_1^2+(x+b)Y_1^3)$, ($a,b \in \mathbb{K}$)
\item (case $d=4$): $Y_0Y_1(Y_0^2 + (ax+b)Y_0Y_1 + Y_1^2)$, ($a \in \mathbb{K}^*$ and $b \in \mathbb{K}$).
\end{itemize}
\end{lemma}
\begin{proof}
We write $F = PQ$.
In order to shorten some expressions, we write $d_P = \deg_Y P$
and $d_P = \deg_Y Q$. We have $d = d_P + d_Q$.
By Theorem \ref{t1} we already have $\deg_x \Delta_Y P \ge d_P-1$.
 The proof splits into different cases according to which case
corresponds to the polynomial $F$ in Lemma \ref{lemma:reducible1}.

Case (0): if $d_Q = 0$. We have
$\deg_x \Delta_Y F = \deg_x \Delta_Y P \ge d-1 
\ge \big\lceil \frac {d-1}2 \big\rceil $.
Equality holds if and only if $d=2$ and $P$ is minimal.
By Theorem \ref{nonmonicintro}, $P$ is $G$--equivalent to a polynomial of the
form $Y_0^2 + (x+c)Y_1^2$, with $c \in \mathbb{K}$.

Case (\ref{case1}): if $d_Q > 0$ and $\deg_x \Delta_Y Q = 0$.
By Lemma \ref{lemma:reducible1}, we can assume that $Q \in \mathbb{K}[Y]$.

Sub-case (\ref{case1}.1): if $d_Q \le d_H-2$. Here, we simply have
$\deg_x \Delta_Y F \ge d_P-1 \ge \frac {d_P+d_Q}2$. In this case, the announced inequality is proved. We then observe that equality implies that
$P$ is minimal, $d_Q = d_P-2$, and $\Res_Y(P,Q)\in \mathbb{K}$.
By Lemma \ref{lemma:minimalresultant1}, this is possible only if
$d_Q=1$ and $d_P=3$. By Theorem \ref{nonmonicintro}, we deduce that
$P$ is $G$--equivalent to a polynomial of the form
$Y_0^3+aY_0Y_1^2+(x+b)Y_1^3$. In this case, $Q$ can only be $Y_1$.

Sub-case (\ref{case1}.2): if $d_Q = d_Q-1$, we have
$\deg_x \Delta_Y F \ge d_P-1 = \frac {d-1}2$. This proves the inequality.
The equality holds if and only if $P$ is minimal and $\Res_Y(P,Q) \in \mathbb{K}$.
By Lemma \ref{lemma:minimalresultant1}, this is possible only if
$d_Q=1$ and $d_P=2$.
By Theorem \ref{nonmonicintro}, we deduce that
$P$ is $G$--equivalent to a polynomial of the form
$Y_0^2+(x+a)Y_1^2$. In this case, $Q$ can only be $Y_1$.

Sub-case (\ref{case1}.3): if $d_Q = d_P$. In this case, we have $d = 2d_P$.
If $P$ is minimal, then by Lemma \ref{lemma:minimalresultant1},
$\deg_x \Res_Y(P,Q) \ge d_P-1$, hence
$\deg_x \Delta_Y F \ge d_P-1 + 2(d_P-1)$. This is always larger than $\frac d2$.
If $P$ is not minimal, we have $\deg_x \Delta_Y P \ge d_P$, whence
the inequalities
$\deg_x \Delta_Y F \ge d_P = \frac d2$. This proves the inequality.
We see here that equality holds only if $\deg_x \Delta_Y P = d_P$ and
$\Res_Y(P,Q) \in \mathbb{K}$. Let $Q = \prod_{i=1}^{d_P} Q_i$ be the factorisation of $Q$
into linear factors in $\mathbb{K}[Y]$, and let $Q_0$ be another linear polynomial 
in $\mathbb{K}[Y]$, coprime to $Q$. We define $R_0 = \Res_Y(P,Q_0) \in \mathbb{K}[x]$.
Using interpolation at the $Q_i$'s, we see that
$P$ can be written as $P = \lambda Q R_0 + b$,
with $\lambda \in \mathbb{K}^*$ and $b \in \mathbb{K}[Y]$. We clearly have
$\deg_x R_0 = \deg_x P = \deg_x F$. We denote by $r_0 \in \mathbb{K}^*$ the leading
coefficient of $R_0$.
$\Delta_Y P$ is an homogeneous polynomial of degree $2(d_P-1)$ in terms
of the coefficients of $P$, hence of degree at most $D=2(d_P-1)\deg_x P$ in $x$.
The coefficient in $x^D$ in its expansion is equal to
$\Disc_Y (\lambda Q r_0)$, which is not zero since $Q$ is squarefree.
This proves that $\deg_x \Delta_Y P = 2(d_P-1)\deg_x P$.
Since this is also equal to $d_P$, the only possibility is $\deg_x P = 1$
and $d_P = 2$. Using the action of $GL_2(\mathbb{K})$, we can therefore assume that
$Q = Y_0Y_1$. Under all these conditions, $P$ is of the form
$P = Y_0^2 + (ax+b)Y_0Y_1 + Y_1^2$, for some $a \in \mathbb{K}^*$ and $b \in \mathbb{K}$.

Sub-case (\ref{case1}.4): if $d_Q \ge d_P+1$. It is impossible for $P$ to have
constant resultants with strictly more than $d_P$ linear polynomials
in $\mathbb{K}[Y]$, since otherwise, by interpolation, it would have coefficients
in $\mathbb{K}$, contradicting its irreducibility. This proves that
$\deg_x \Res_Y(P,Q) \ge d_Q - d_P$. We then have the inequalities
$\deg_x \Delta_Y F \ge d_P-1 + 2(d_Q-d_P) \ge d_Q \ge \frac {d+1}2 $.
This proves the announced inequality and in this case an equality is impossible.

Cases (\ref{case2}) and (\ref{case3}): in the remaining cases,
we have $d_Q \ge 2$ and $\deg_x \Delta_Y Q > \frac {d_Q}2$.
This gives
$\deg_x \Delta_Y F > d_P-1 + \frac {d_Q}2
\ge \frac {d_P}2 + \frac {d_Q}2 = \frac d2$, whence the conclusion.
\end{proof}

\begin{lemma}\label{lemma:resultant1}
Let $q = y^2+ay+b$ be a polynomial in $\mathbb{K}[x][y]$, with
$a$ and $b$ in $\mathbb{K}[x]$. Assume that $\deg_x a^2-4b$ is odd.

For a polynomial $p \in \mathbb{K}[x][y]$, we have
$ \Res_y(p,q) \in \mathbb{K} $ if and only if
$p = \alpha q + \beta$ for some $\alpha \in \mathbb{K}[x][y]$ and $\beta \in \mathbb{K}$.
\end{lemma}

\begin{proof}
Let $p = \alpha q + uy + v$ be the euclidean division of $p$ by $q$, with
$u$ and $v$ in $\mathbb{K}[x]$. We have
$ \Res_y(p,q) = \Res_y (uy+v,q) = (v-au/2)^2-\frac {a^2-4b}{4} u^2$.
By assumption, this is an element of $\mathbb{K}$.
Since $\deg_x a^2-4b$ is odd, inspecting degrees shows that this is
possible only if $u = 0$ and $v-au/2 \in \mathbb{K}$. This gives the conclusion.
\end{proof}

We are now ready to prove Theorem \ref{treducible} that we reformulate in a more convenient form for the proof:

\begin{theorem}\label{t4}
Let $F \in \mathbb{K}[x][Y]$ be a squarefree polynomial of degree $\deg_Y F = d \ge 0$
with no factor in $\mathbb{K}[x]$.
Then exactly one of the following occurs:
\begin{enumerate}
\item $\deg_x \Delta_Y F = 0$ and $F$ is $G$--equivalent to some polynomial of $\mathbb{K}[Y]$.
\item $d\ge 2$ and $\deg_x \Delta_Y F \ge \big\lceil \frac {d-1}2 \big\rceil$.
\end{enumerate}

Furthermore, if $d\ge 2$, equality $\deg_x \Delta_Y F = \big\lceil \frac {d-1}2 \big\rceil$ occurs if and only if $F$ is $G$--equivalent
to one of the following polynomials:
\begin{itemize}
\item (case $d$ odd): $Y_1 \prod_{i=1}^n (Y_0^2+(x+a_i)Y_1^2)$ ($a_i \in \mathbb{K}$).
\item (case $d$ even): $~~\prod_{i=1}^n (Y_0^2+(x+a_i)Y_1^2)$ ($a_i \in \mathbb{K}$).
\item (case $d=4$): $Y_1(Y_0^3+aY_0Y_1^2+(x+b)Y_1^3)$ ($a,b \in \mathbb{K}$)
\item (case $d=4$): $Y_0Y_1(Y_0^2 + (ax+b)Y_0Y_1 + Y_1^2)$ ($a \in \mathbb{K}^*$ and $b \in \mathbb{K}$).
\end{itemize}
\end{theorem}
\begin{proof}
Write $F = PQ$ where $Q$ has only linear factors and $P$ has no linear factor.
Let $P = \prod_{i=1}^n P_i$ be the decomposition of $P$ into irreducible factors in $\mathbb{K}[x][Y]$.
If $n = 0$ then the result is given by Lemma \ref{lemma:reducible1}.
Assume now that $n\ge 1$.
The polynomial $F_1 = P_1Q$ satisfies Lemma \ref{lemma:reducible2}, hence
$\deg_x \Delta_Y F_1 \ge \frac {\deg Q + \deg P_1-1}2$.
For $i\ge 2$, the polynomials $P_i$ satisfy Theorem \ref{t1}, hence
$\deg_x \Delta_Y P_i \ge \deg P_i-1 \ge \frac {\deg P_i}2$.
Putting these inequalities together gives
\begin{equation}\label{eq1}
\deg_x \Delta_Y F \ge \frac {\deg Q + \deg P_1-1}2 + \sum_{i\ge 2} \frac {\deg P_i}2
= \frac {d-1}2
\end{equation}
Consider now the question of equality.
The easiest case is when $d$ is odd. In this situation, all inequalities
in (\ref{eq1}) are equalities. This implies that $\deg P_i = 2$ for all
$i\ge 2$ and $\deg F_1$ is odd with $\deg_x \Delta_Y F_1 = \frac {\deg F_1-1}2$.
By Lemma \ref{lemma:reducible2}, we can therefore assume that
$F_1 = Y_1 (Y_0^2+(x+a_1)Y_1^2)$.
The $P_i$'s have constant resultant with $Y_1$ and $Y_0^2+(x+a_1)Y_1^2$.
Using Lemma \ref{lemma:resultant1}, we deduce that they
are of the form $P_i = b_i (Y_0^2 + (x+a_i)Y_1^2)$ with $a_i,b_i \in \mathbb{K}$. The constant $\prod b_i$ can be removed using $G$--equivalence.
This gives the conclusion for $d$ odd.

If $d$ is even, Lemma \ref{lemma:reducible2} shows that
$F$ can not have more that $2$ linear factors. We have therefore three cases
to consider:

$\bullet$ If $F$ has no linear factor, then by \ref{lemma:reducible2}
we can assume that $P_1 = Y_0^2+(x+a_1)Y_1^2$ for some $a_1 \in\mathbb{K}$. 
The proof in this case is very similar to the previous case and left to
the reader.

$\bullet$ If $F$ has one linear factor, then by Lemma \ref{lemma:reducible2},
it is enough to consider the case $F_1 = Y_1P_1$ with 
$P_1 = Y_0^3+aY_0Y_1^2+(x+b)Y_1^3$ for some $a,b \in \mathbb{K}$.
The other factors $P_i$ must be quadratic and minimal, and
also have constant resultant with $F_1$. The resultant with $Y_1$ shows that
the $P_i$'s are monic in $Y_0$. If $n\ge 2$, the resultant of
$P_2$ and $P_1=Y_0^3+aY_0Y_1^2+(x+b)Y_1^3$ is constant, and
Lemma \ref{lemma:resultant1} imposes that $P_1 = Y_0 P_2 + \beta Y_1^3$ with
$ \beta \in \mathbb{K}$. This is incompatible with $\beta = x+b$, hence we must deduce
that $n = 1$ and $F = F_1$.

$\bullet$ If $F$ has two linear factors, then by Lemma \ref{lemma:reducible2},
it is enough to consider the case $F_1 = Y_0Y_1P_1$ with
$P_1=Y_0^2 + (ax+b)Y_0Y_1 + Y_1^2$ for some $a \in \mathbb{K}^*$ and $b \in \mathbb{K}$.
The other factors $P_i$ must be
quadratic and minimal, and also have constant resultant with $F_1$.
In particular, if $n\ge 2$, $\Res_Y(Y_0Y_1,P_2)\in \mathbb{K}$ imposes that
$P_2 = a_2Y_0^2+b_2Y_0Y_1+c_2Y_1^2$ with $a_2$ and $c_2$ in $\mathbb{K}$. But this
is incompatible with $\deg_x \Delta_Y P_2 = 1$, hence we must deduce
that $n = 1$ and $F = F_1$.
\end{proof}

\appendix

\section{Small $\Delta_y$ versus small $\Delta_x$}\label{s3}

The equivalence $(a)\Leftrightarrow (b)$ of Theorem \ref{t3} asserts that a monic polynomial is minimal with respect to $y$ if and only it is monic and minimal with respect to $x$. We prove here a generalisation of this statement to the case of nonmonic polynomials.

\vskip2mm
\noindent

For $f\in \mathbb{K}[x,y]$ a nonconstant bivariate polynomial we let $\lc_y(f)$ (resp. $\lc_x(f)$) stand for the leading coefficient of $f$ seen as a polynomial in $y$ (resp. in $x$). We denote by $n_x$ (resp. $n_y$) the number of \textit{distinct roots} of $\lc_y$ (resp. of $\lc_x$). We have the inequalities 
\[
n_x\le \deg_x \lc_y(f)\quad {\rm and} \quad n_y\le \deg_y \lc_x(f)
\] 
and we say that $f$ is \textit{nondegenerate} if both equalities hold, that is if both leading coefficients of $f$ are squarefree. We write for short $f(\infty,\infty)=0$ if the bihomogenisation $F$ of $f$ vanishes at the point $X_1=Y_1=0$, that is if $f$ has no monomial of bidegree $(d_x,d_y)$.

\begin{proposition}\label{symmetry}
Let $f\in \mathbb{K}[x,y]$ be a nondegenerate irreducible bivariate polynomial such that $f(\infty,\infty)=0$. The following assertions are equivalent:
\begin{enumerate}[(a)]
\item $\deg_x \Delta_y(f)=d_y+n_y-1$.
\item $\deg_y \Delta_x(f)=d_x+n_x-1$.
\item The Zariski closure $C\subset \mathbb{P}^1\times \mathbb{P}^1$ of the affine curve $f=0$ is rational, unicuspidal and smooth outside $(\infty,\infty)$.
\end{enumerate}
Moreover, the equivalence $(c)\Leftrightarrow (a)\cap (b)$ still holds for degenerate polynomials.
\end{proposition}

\begin{proof} Let us first prove $(c)\Leftrightarrow (a)\cap (b)$. Hence $f$ is allowed to be degenerate.

$\bullet\,\,(c)\Rightarrow (a)\cap (b).$ By Proposition \ref{p1}, we have the equality
\[
\deg_x \Delta_y(f)=2d_x(d_y-1)-2\delta_{\infty}-r_{\infty}
\]
where $\delta_{\infty}$ and $r_{\infty}$ stand respectively for the delta invariant and the ramification index of $f$ over $x=\infty$. Since $C$ is assumed to be rational with a unique possible singularity at $(\infty,\infty)$, the adjunction formula leads to the equality
\[
\delta_{\infty}=p_a(C)=(d_x-1)(d_y-1).
\]
Moreover, the curve is assumed to be everywhere locally irreducible. Hence the number of places of $C$ over $x=\infty$ coincides with the number of intersection points of $C$ with $x=\infty$, that is $n_x+1$. It follows that
\[
r_{\infty}=d_y-(n_x+1).
\] 
Equality $(a)$ then follows from Proposition \ref{p1}. The implication 
$(c)\Rightarrow (b)$ follows from $(c)\Rightarrow (a)$ by symmetry.

$\bullet\,\,(a)\cap (b)\Rightarrow (c)$. Let us assume that $(a)$ holds. By Proposition \ref{p1}, we have:
\begin{equation}\label{delta}
2\delta_{\infty}=2d_x(d_y-1)-(d_y+n_x-1)-r_{\infty}
\end{equation}
By assumption, the curve $C$ of $f$ has at least $n_y+1$ places over $x=\infty$ so that
\[
r_{\infty}\le d_y-n_y-1.
\]
Combined with (\ref{delta}), we get the inequality
\[
2\delta_{\infty}\ge 2(d_x-1)(d_y-1).
\]
On  the other hand, the genus being nonnegative, the adjunction formula leads to the inequality
\[
\delta_{\infty}= p_a(C)-g\le p_a(C)=(d_x-1)(d_y-1).
\]
This forces $\delta_{\infty}=p_a(C)$. Hence $g=0$ and the singularities of $C$ are located along the line $x=\infty$.  This forces also $r_{\infty}=d_y-n_y-1$ so that the curve $C$ has exactly $n_y+1$ places over $x=\infty$, hence is locally irreducible along the line $x=\infty$. If moreover $(b)$ holds, we get by symmetry that $C$ has all its singularities located on the line $y=\infty$, and that $C$ has exactly $n_x+1$ places over $y=\infty$. Hence $(a)\cap (b)$  forces $C$ to be rational, with a unique possible singularity at $(\infty,\infty)$, this singularity being irreducible.

To finish the proof, we need to show that implication $(a)\Rightarrow (c)$ holds when $f$ is nondegenerate. We just proved that $(a)$ implies that $C$ is rational with all its singularities irreducible and located on the line $x=\infty$. The nondegenerate assumption ensures that $C$ is transversal to the line $x=\infty$ (hence smooth) except possibly at $(\infty,\infty)$. Hence $(c)$ holds. 
\end{proof}

\begin{corollary}\label{minimaldisc}
Let $f\in \mathbb{K}[x,y]$ be an irreducible bivariate polynomial such that $f(\infty,\infty)=0$. Then 
\[
\deg_x \Delta_y(f)=d_y-1 \Longrightarrow \begin{cases} \deg_y \Delta_x(f)=d_x+n_x-1 \\ 
n_y=0 \end{cases}
\]
and the converse holds for nondegenerate polynomials. In particular, polynomials vanishing at $(\infty,\infty)$ and minimal with respect to $y$ are monic with respect to $x$.
\end{corollary}

\begin{proof}
If $f$ is minimal, its curve $C\subset \mathbb{P}^1\times \mathbb{P}^1$ is rational unicuspidal with a unique place on $x=\infty$ by Theorem \ref{t1}. This place has to be $(\infty,\infty)$ by assumption. This forces $n_y=0$. The equality $\deg_y \Delta_x(f)=d_x+n_x-1$ follows from Proposition \ref{symmetry}. If $f$ is nondegenerate, the converse holds again by Proposition \ref{symmetry}.
\end{proof}

\section{Parametrisation of minimal polynomials.}\label{ssparam}

Let $f=\sum \alpha_{ij} x^i y^j\in \mathbb{K}[x,y]$ be a polynomial with parameters $(d_x,d_y,c)$ and with indeterminate coefficients 
\[
\alpha=(\alpha_{ij})_{(i,j)\in P(f)\cap \mathbb{Z}^2}.
\]
The  discriminant of $f$ is a polynomial in $(x,\alpha)$ of degree $2d_x(d_y-1)$ in $x$. Thus, in order to find which specialisations of $\alpha$ lead to a minimal polynomial, one needs to compute $\Delta_y(f)$ and then to solve a system of 
\[
2d_x(d_y-1)-(d_y-1)\in \mathcal{O}(d_x d_y)
\] 
polynomial equations in $\alpha$ with 
\[
\Card(P(f)\cap \mathbb{Z}^2)\in \mathcal{O}(d_x d_y)
\]
unknowns. This polynomial system turns out to be very quickly too complicated to be solved on a computer, even for reasonable size of $d_x$ and $d_y$. Moreover, there remains to perform an irreducibility test for each solution. However, we know that minimal polynomials define a rational curve, a strong information that is not used in the previous basic strategy. 
In particular, the curve admits a rational parametrisation, that is to say there exist two rational functions $u,v\in \mathbb{K}(s)$ such that the equality 
\[
f(u(s),v(s))=0
\]
holds in $\mathbb{K}(s)$. The following result summarises the relations between  minimality and parametrisation.

\begin{proposition}\label{paramet}
An irreducible polynomial $f\in \mathbb{K}[x,y]$ is minimal if and only if there exist two rational functions $u,v\in \mathbb{K}(s)$ such that:
\begin{enumerate}
\item $f(u,v)=0$ in $\mathbb{K}(s)$ (rationality)
\item $\mathbb{K}(s)=\mathbb{K}(u,v)$ (proper parametrisation)
\item $u\in \mathbb{K}[s]$ (unique place along $x=\infty$)
\item $\mathbb{K}[s]= \mathbb{K}[u,v]\cap \mathbb{K}[u,v^{-1}]$ (smoothness in $\mathbb{A}^1\times \mathbb{P}^1$).
\end{enumerate}
Moreover, given such a pair $u,v$, we have the equality
\[
d_y=\deg_s u,\quad a(f)=\deg_s v_1\quad and \quad b(f)=\deg_s v_2
\]
where $v_1,v_2\in \mathbb{K}[t]$ are coprime polynomials such that $v=v_1/v_2$. 
\end{proposition}

\begin{proof}
We know by Theorem \ref{t1} that $f$ is minimal if and only if the curve $C\subset \mathbb{P}^1\times \mathbb{P}^1$ is rational, with a unique place along $x=\infty$ and smooth outside this line.
Rationality is equivalent to the existence of a proper parametrisation, that is the existence of rational functions $u,v\in \mathbb{K}(s)$ such that items $(1)$ and $(2)$ hold. The rational map 
\[
(u,v):\mathbb{K} \dashrightarrow \mathbb{K}^2
\]
extends to a morphism 
\[
\rho: \mathbb{P}^1 \rightarrow  \mathbb{P}^1\times \mathbb{P}^1 
\]
whose image is $C$. Moreover, the parametrisation being proper, this morphism establishes a one-to-one correspondence between $\mathbb{P}^1$ and the places of $C$. The fact that $C$ has a unique place along the line $x=\infty$ is equivalent to the fact that $u$ as a unique pole on $\mathbb{P}^1$. Up to a Moebius transformation on $\mathbb{P}^1$, there is no less to assume that this pole is $s=\infty$, meaning precisely that $u\in \mathbb{K}[s]$. The  restriction of $C$ to $\mathbb{A}^1\times \mathbb{P}^1$ is smooth if and only if its restrictions to the two affine charts $U:=\mathbb{A}^1\times \{y\ne \infty\} \simeq  \mathbb{A}^2$ and $V:=\mathbb{A}^1\times \{y\ne 0\} \simeq  \mathbb{A}^2$. But this is also equivalent to the fact that the coordinate rings 
\[
\frac{\mathbb{K}[x,y]}{(f(x,y))}\simeq \mathbb{K}[u,v]	\quad {\rm and} \quad \frac{\mathbb{K}[x,y]}{(y^{d_y}f(x,1/y))}\simeq \mathbb{K}[u,v^{-1}]
\]  
of the affine curves $C_{|U}$ and $C_{|V}$ are integrally closed in their field of fractions $\mathbb{K}(u,v)=\mathbb{K}(s)$. Since $u\in \mathbb{K}[s]$, we deduce that $s$ is integrally closed over $\mathbb{K}[u,v]$ and over $\mathbb{K}[u,v^{-1}]$. Hence, an inclusion $\mathbb{K}[s]\subset \mathbb{K}[u,v]\cap \mathbb{K}[u,v^{-1}]$. The reverse inclusion always holds by a Gauss Lemma argument, and we get item $(4)$. Conversely, if item $(4)$ holds, then $\mathbb{K}[u,v]\cap \mathbb{K}[u,v^{-1}]=\mathbb{K}[s]$ is integrally closed so that the curve is smooth in $\mathbb{A}^1\times \mathbb{P}^1$.  The formulas for $\deg_y f$, $a(f)$ and $b(f)$ follow for instance from \cite{DS}, where the authors compute the Newton polytope of a parametrised curve.  
\end{proof}

\paragraph*{Computation of minimal polynomials.} Thanks to Proposition \ref{paramet}, computing all minimal polynomials of given parameters $(d_x,d_y,c)$ is equivalent to computing the discriminant of the  implicit equation of the parametrisation $(u,v)$ with indeterminate coefficients that satisfies items $(1), (2), (3)$ and solving a system of 
\[
(2d_x-1)(d_y-1)
\in \mathcal{O}(d_x d_y)
\] 
polynomial equations with 
\[
d_y+d_x+c\in \mathcal{O}(d_x + d_y)
\]
unknowns. When compared to the previous approach, we reduce drastically the number of unknwons and we avoid the irreducibility tests. This is the approach which allowed us to find the crucial example of Theorem \ref{contrexintro} by computer. It has to be noticed however that the degree of the polynomial system then increases.
Finally, let us mention that item $(4)$ (hence minimality) can also be checked directly by requiring that the so-called $D$-resultant of the pair $(u,v)$ is constant \cite{gut}, a computational problem of an {\it a priori} equivalent complexity.

\section{Coordinate polynomials are minimal}\label{C}

We found it instructive to give a direct proof of $(d)\Rightarrow (a)$ in Theorem \ref{t3} (coordinate polynomials are minimal) that does not use the embedding line theorem  of Abhyankar-Moh.

We recall that $\Aut(\mathbb{A}^2)$ is the set of polynomial automorphisms of the affine plane, that is maps $\sigma = (\sigma_x,\sigma_y)$,
where $\sigma_x$ and $\sigma_y$ are polynomials of $\mathbb{K}[x,y]$, such that there exists another
$\sigma^{-1} = (\sigma_x^{-1} , \sigma_y^{-1})$,
with $\sigma_x^{-1}$ and $\sigma_y^{-1}$ also elements of $\mathbb{K}[x,y]$, satisfying the relations
\[ \sigma \circ \sigma^{-1} = \sigma^{-1}\circ \sigma = \id.\]
As is easily seen, $\Aut(\mathbb{A}^2)$ is a group for the composition.

\begin{lemma}\label{lemma:irreducible}
For $\sigma \in \Aut(\mathbb{A}^2)$, the polynomials
$\sigma_x$ and $\sigma_y$ are irreducible in $\mathbb{K}[x,y]$.
\end{lemma}
\begin{proof}
Assume $\sigma_x = fg$ for some $f,g \in \mathbb{K}[x,y]$.
Then $x = \sigma_x \circ \sigma^{-1} = f\circ \sigma^{-1} \times g\circ \sigma^{-1}$.
But $x$ is irreducible and $f\circ \sigma^{-1}$ and $g\circ \sigma^{-1}$ are polynomials.
Hence one of them is constant, say $f \circ \sigma^{-1}$. Composing again with $\sigma$, we deduce
that $f$ itself is constant. Hence $\sigma_x$ is irreducible.
From the relation $y = \sigma_y \circ \sigma^{-1}$, we also deduce that $\sigma_y$ is irreducible.
\end{proof}

\begin{lemma}\label{lemma:stillirreducible}
For $\sigma \in \Aut(\mathbb{A}^2)$, and any $u,v \in \mathbb{K}$,
the polynomials $\sigma_x-u$ and $\sigma_y-v$ are irreducible in $\mathbb{K}[x,y]$.
\end{lemma}
\begin{proof}
Let us define $\tau = (x-u , y-v)$.
Clearly, $\tau$ and $\tau \circ \sigma$ are in $\Aut(\mathbb{A}^2)$.
The conclusion is given by Lemma \ref{lemma:irreducible} applied to $\tau \circ \sigma$.
\end{proof}

\begin{remark}
We know that the conclusion of Lemma
\ref{lemma:stillirreducible} also holds for $\sigma^{-1}$. Hence $\sigma_x^{-1}-u$, and $\sigma_y^{-1}-v$ are irreducible polynomials.
\end{remark}

For the next lemma, we introduce the jacobian of $\sigma$, which is defined as
\[ J_\sigma = 
  \begin{pmatrix}
    \frac{\partial \sigma_x}{\partial x} & \frac{\partial \sigma_y}{\partial x} \\
    \frac{\partial \sigma_x}{\partial y} & \frac{\partial \sigma_y}{\partial y} \\
  \end{pmatrix} \; .
\]

\begin{lemma}\label{lemma:jacobian}
Let $\sigma$ be an element of $\Aut(\mathbb{A}^2)$.
There exists some $\lambda \in \mathbb{K}^*$ such that
\[ \det J_\sigma = \lambda \]
and
\[   
  \begin{pmatrix}
    \frac{\partial \sigma_x}{\partial x} & \frac{\partial \sigma_y}{\partial x} \\
    \frac{\partial \sigma_x}{\partial y} & \frac{\partial \sigma_y}{\partial y} \\
  \end{pmatrix}
= \lambda
  \begin{pmatrix}
    \frac{\partial \sigma_x^{-1}}{\partial x} \circ \sigma & \frac{\partial \sigma_y^{-1}}{\partial x} \circ \sigma \\
    \frac{\partial \sigma_x^{-1}}{\partial y} \circ \sigma & \frac{\partial \sigma_y^{-1}}{\partial y} \circ \sigma \\
  \end{pmatrix}
\]
\end{lemma}
\begin{proof}
Differentiating the relation $\sigma^{-1} \circ \sigma = \id$ gives
\[ J_\sigma\times (J_{\sigma^{-1}}\circ \sigma) = \begin{pmatrix} 1&0\\0&1 \end{pmatrix}\]
Each matrix involved in this relation has polynomial coefficients, hence the determinant of the 
jacobians must be nonzero constants. The announced relation among $J_\sigma$ and
$J_{\sigma^{-1}}\circ\sigma$ is then given by the classical formula for the inverse of a $2\times 2$ matrix.
\end{proof}

\begin{lemma}\label{lemma:Phi_xconstant}
Let $\sigma$ be an element of $\Aut(\mathbb{A}^2)$. The following properties are equivalent:
\begin{enumerate}
\item There exists some $x_0 \in \mathbb{K}$ such that $\sigma_x(x_0,y)$ is a constant polynomial
\item There exists some $x_0 \in \mathbb{K}$ such that $\sigma_x^{-1}(x_0,y)$ is a constant polynomial
\item There exist some $a\in \mathbb{K}^*$ and $b \in \mathbb{K}$ such that $\sigma_x = ax+b$
\item There exist some $a\in \mathbb{K}^*$ and $b \in \mathbb{K}$ such that $\sigma_x^{-1} = a^{-1}x-a^{-1}b$
\item $\deg_y \sigma_x = 0$
\item $\deg_y \sigma_x^{-1} = 0$
\end{enumerate}
\end{lemma}
\begin{proof}

$(1)\Rightarrow(3)$: Let $c = \sigma_x(x_0,y)$.
Then we have $\sigma_x = (x-x_0)f + c$ for some $f \in \mathbb{K}[x,y]$.
This implies $\sigma_x -c = (x-x_0)f $, which must be irreducible by Lemma \ref{lemma:irreducible}.
Hence $f$ is constant, say $f = a$. We have now $\sigma_x = ax + b$ for some $b\in \mathbb{K}$. Because
$\sigma \circ \sigma^{-1} = \id$, $\sigma_x$ must be surjective, hence $a\neq 0$.

$(3)\Rightarrow(5)$ and $(5)\Rightarrow(1)$ are trivial.

Because $\sigma^{-1}$ is also in $\Aut(\mathbb{A}^2)$, we also have
$(2) \Leftrightarrow (4) \Leftrightarrow (6)$.

$(3)\Rightarrow(4)$:
From the relation $\sigma \circ \sigma^{-1} = \id$, we deduce
$a \sigma_x^{-1} + b = x$, whence the conclusion.

$(4)\Rightarrow(3)$: similar to $(3)\Rightarrow(4)$.
\end{proof}

\begin{lemma}\label{lemma:genericdisc_p0}
Let $f = f_dy^d + \dots + f_0$ be a generic polynomial of degree $d>0$, and
consider its discriminant $\Delta = \Disc f$.
As a polynomial in $f_0$, we have
\[\Delta = (-1)^{d(d-1)/2} d^d f_d^{d-1} f_0^{d-1} + o(f_0^{d-1}) \; .\]
\end{lemma}
\begin{proof}
We have $(-1)^{d(d-1)/2} f_d \Delta = \det S$,
where $S$ is the Sylvester matrix of $f$ and $f'$:
\[ S = 
  \left( \begin{array}{ccc|cccc}
   f_d &        &    & df_d &        &  &  \\
       & \ddots &    &      & \ddots &  &  \\
       &        & f_d&      &        & df_d &\\
\hline
   f_1 &        &    & f_1   &        &      & df_d\\
\hline
   f_0&         &    &       & \ddots &      &\\
        & \ddots& f_1&       &        & f_1  &\\
        &       & f_0&       &        &      & f_1\\
    \end{array} \right)
\]
Expanding this determinant gives
\[
 \det S 
  = (-1)^{d(d-1)} f_0^{d-1} (df_d)^d + o(f_0^{d-1})
\]
whence the result.
\end{proof}

\begin{proposition}\label{prop:equality_degrees}
Let $\sigma$ be an element of $\Aut(\mathbb{A}^2)$.
The following equalities hold:
\[ \begin{array}{rl}
  \deg_y \sigma_x = \deg_y \sigma_x^{-1} & \qquad
  \deg_x \sigma_y = \deg_x \sigma_y^{-1}\\

  \deg_x \sigma_x = \deg_y \sigma_y^{-1} & \qquad 
  \deg_y \sigma_y = \deg_x \sigma_x^{-1}\\
\end{array}\]

If $\deg_y \sigma_x = 0$, then $\sigma_x = ax+b$, otherwise $\sigma_x$ is monic in $y$.

If $\deg_x \sigma_y = 0$, then $\sigma_y = ay+b$, otherwise $\sigma_y$ is monic in $x$.

If $\deg_x \sigma_x = 0$, then $\sigma_x = ay+b$, otherwise $\sigma_x$ is monic in $x$.

If $\deg_y \sigma_y = 0$, then $\sigma_y = ax+b$, otherwise $\sigma_y$ is monic in $y$.

\end{proposition}
\begin{proof}
We focus on the first equality and the first sentence.

$\bullet$
If $\deg_y \sigma_x = 0$ or $\deg_y \sigma_x^{-1} = 0$, then by Lemma \ref{lemma:Phi_xconstant},
we have the equality.
We assume now that $\deg_y \sigma_x > 0$ and $\deg_y \sigma_x^{-1} > 0$.
We write $d = \deg_y \sigma_x$ and $D = \deg_y \sigma_x^{-1}$.

$\bullet$
For $x_0\in \mathbb{K}$ and $u_0 \in \mathbb{K}$,
we write $d_0 = \deg_y \sigma_x(x_0,y)$ and
$D_0 = \deg_v \sigma_x^{-1}(u_0,v)$.
By Lemma \ref{lemma:Phi_xconstant}, we know that $0 < d_0 \le d$
and $0 < D_0 \le D$.
The number of distinct roots of the equation $\sigma_x(x_0,y)-u_0$ is an integer $r$ such that
$0 < r < d \le d_0$. 
We denote by $(y_i)_{i\le r}$ these distinct roots.
Similarly, the number of distinct roots of the equation
$\sigma_x^{-1}(u_0,v)-x_0$ is an integer $R$ such that 
$0 < R < D \le D_0$.
We have $\sigma(x_0,y_i) = (u_0,\sigma_y(x_0,y_i))$.
Because $\sigma$ is a bijection on $\mathbb{A}^2$ and the $y_i$ are distinct, we deduce that
the $v_i=\sigma_y(x_0,y_i)$ are also distinct. Hence $r \le R$.
Of course, $x_0$ and $u_0$ play a symmetrical role, hence we also have $R \le r$, whence
$r = R$.

$\bullet$
Now, $x_0 \in \mathbb{K}$ is still fixed without any condition, but we assume that $u_0 \in \mathbb{K}$
is chosen such that the polynomial $\sigma_x(x_0,y) - u_0$ is squarefree
(this is possible since $\deg_y \sigma_x(x_0,y) = d_0 > 0$, hence by Lemma
\ref{lemma:genericdisc_p0}, $\Disc_y (\sigma_x(x_0,y)-u_0)$ is not the zero polynomial).
We also assume that $u_0$ is chosen such that
$\deg_v \sigma_x^{-1}(u_0,v)-x_0 = D$ and
$\sigma_x^{-1}(u_0,v) - x_0$ is squarefree (it would be impossible only if
$\Disc_v(\sigma_x^{-1}(u,v) - x_0)$ were the zero polynomial, hence
$\sigma_x^{-1}(u,v)-x_0$ would have a square factor. This is however not allowed by
Lemma \ref{lemma:stillirreducible}).
With all these assumptions, we have the following equalities:
$r = d_0$ and $R = D_0 = D$. Since we have proved $r=R$, we have
$d_0 = D$.

$\bullet$
There is certainly a choice of $x_0$ such that
$\deg_y \sigma(x_0,y) = \deg_y \sigma(x,y)$, hence we have $d = D$, which is the
announced equality. Of course, this equality does not depend an any choice of $x_0$ and $u_0$.

$\bullet$
If $x_0$ is chosen without any condition, we still have the equalities
$d_0 = D = d$. This imply that $\deg_y \sigma_x(x_0,y)$ does not
depend on $x_0$, hence the leading coefficient of $\sigma_x(x,y)$ along the
variable $y$ does not depend on $x$. This proves that $\sigma_x(x,y)$ is monic in $y$.

$\bullet$
The first equality and the first sentence are proved, and we can now turn to
the others. They are obtained from the first case by considering 
successively $S \circ \sigma \circ S$, $\sigma \circ S$, and $S \circ \sigma$,
where $S(x,y) = (y,x)$.
\end{proof}

\begin{theorem}
Let $\sigma$ be an element of $\Aut(\mathbb{A}^2)$.
If $\deg_y (\sigma_x) \neq 0$, then
\[ \deg_x (\Disc_y \sigma_x) = \deg_y (\sigma_x) - 1 \;.\]
\end{theorem}

\begin{proof}By Proposition \ref{prop:equality_degrees}, $\sigma_x$ is monic in $y$. Let $a \in \mathbb{K}^*$ be its leading
coefficient in $y$, and $d$ its degree. For $x\in \mathbb{K}$ fixed, we have
\[\begin{array}{rll}
 (-1)^{d(d-1)/2} a \Disc_y \sigma_x 
 &= \Res_y(\sigma_x(x,y), \frac{\partial \sigma_x}{\partial y}(x,y)) \\
 &= \displaystyle a^{d-1} \prod_{y,\sigma_x(x,y)=0} \frac{\partial \sigma_x}{\partial y}(x,y)\\
\end{array}\]
Since $x \in \mathbb{K}$ is fixed, the equation $\sigma_x(x,y) = 0$
is equivalent to $\sigma(x,y) = (0,t)$.
Hence the solutions of $\sigma_x(x,y) = 0$ are given by
$y = \sigma_y^{-1}(0,t)$, where $t$ are the solutions of $x = \sigma_x^{-1}(0,t)$.
We have therefore
\[\begin{array}{rll}
 (-1)^{d(d-1)/2} a \Disc_y \sigma_x 
 &= \displaystyle a^{d-1} \prod_{t,\sigma_x^{-1}(0,t)=x} \frac{\partial \sigma_x}{\partial y}(\sigma^{-1}(0,t))\\
 &= \displaystyle a^{d-1} \prod_{t,\sigma_x^{-1}(0,t)=x} \lambda^{-1} \frac{\partial \sigma_x^{-1}}{\partial y}(0,t) & \qquad \qquad \text{(by Lemma \ref{lemma:jacobian})}\\
\end{array}\]
But $\sigma_x^{-1}(x,y)$ is monic in $y$ of degree $d$ (by Proposition \ref{prop:equality_degrees}).
We denote by $b\in \mathbb{K}^*$ its leading coefficient.
The degree of the polynomial $\sigma_x^{-1}(0,t)$ in $t$ is therefore $d$, and its leading coefficient is $b$.
We can now write
\[\begin{array}{rll}
 (-1)^{d(d-1)/2} a \Disc_y \sigma_x 
 &= \displaystyle a^{d-1} \lambda^{-d} b^{1-d} \Res_t \left(\sigma_x^{-1}(0,t)-x, \frac{\partial \sigma_x^{-1}}{\partial y}(0,t)\right)
\end{array}\]
This equality is true for all $x\in \mathbb{K}$, hence is a polynomial equality in $\mathbb{K}[x]$.
We also have
\[ \Disc_y \sigma_x
 = a^{d-2} \lambda^{-d} b^{2-d} \Disc_t (\sigma_x^{-1}(0,t)-x)\]
Since $\sigma_x^{-1}(0,t)$ has degree exactly $d$ in $t$, by Lemma \ref{lemma:genericdisc_p0},
$\Disc_y \sigma_x$ has degree exactly $d-1$ in $x$.
\end{proof}

\affiliationone{
   D. Simon and M. Weimann\\
   LMNO, Universit\'e de Caen BP 5186\\
   F 14032 Caen Cedex\\
   France
   \email{denis.simon@unicaen.fr}\\
   \email{martin.weimann@unicaen.fr}
}

\end{document}

%% file: polydisc2.pdf_t
\begin{picture}(0,0)%
\includegraphics{polydisc2.pdf}%
\end{picture}%
\setlength{\unitlength}{4144sp}%
\begingroup\makeatletter\ifx\SetFigFont\undefined%
\gdef\SetFigFont#1#2#3#4#5{%
  \reset@font\fontsize{#1}{#2pt}%
  \fontfamily{#3}\fontseries{#4}\fontshape{#5}%
  \selectfont}%
\fi\endgroup%
\begin{picture}(21402,6147)(-14,-7546)
\put(  1,-3661){\makebox(0,0)[lb]{\smash{{\SetFigFont{34}{40.8}{\familydefault}{\mddefault}{\updefault}{\color[rgb]{0,0,0} $d_y$}%
}}}}
\put(11476,-3661){\makebox(0,0)[lb]{\smash{{\SetFigFont{34}{40.8}{\familydefault}{\mddefault}{\updefault}{\color[rgb]{0,0,0} $d_y$}%
}}}}
\put(4501,-7486){\makebox(0,0)[lb]{\smash{{\SetFigFont{34}{40.8}{\familydefault}{\mddefault}{\updefault}{\color[rgb]{0,0,0}$c=b$}%
}}}}
\put(7426,-7486){\makebox(0,0)[lb]{\smash{{\SetFigFont{34}{40.8}{\familydefault}{\mddefault}{\updefault}{\color[rgb]{0,0,0} $d_x=a$}%
}}}}
\put(15976,-7486){\makebox(0,0)[lb]{\smash{{\SetFigFont{34}{40.8}{\familydefault}{\mddefault}{\updefault}{\color[rgb]{0,0,0}$c=a$}%
}}}}
\put(18901,-7486){\makebox(0,0)[lb]{\smash{{\SetFigFont{34}{40.8}{\familydefault}{\mddefault}{\updefault}{\color[rgb]{0,0,0}  $d_x=b$}%
}}}}
\end{picture}%